\documentclass[a4paper,12pt]{article}
\usepackage[T1]{fontenc}
\usepackage[utf8]{inputenc}
\usepackage{lmodern}
\usepackage{amssymb,amsmath,mathrsfs,graphicx,floatrow,subfigure,theorem}
\usepackage{enumerate,tikz,txfonts}

\let\oldsection\section
\renewcommand\section{\setcounter{equation}{0}\oldsection}
\renewcommand\thesection{\arabic{section}}
\renewcommand\theequation{\thesection.\arabic{equation}}

\newtheorem{lemma}{Lemma}[section]
\newtheorem{theorem}{Theorem}[section]
\newtheorem{remark}{Remark}[section]
\newtheorem{proposition}{Proposition}[section]
\newtheorem{corollary}{Corollary}[section]

\allowdisplaybreaks

\begin{document}
\title{\bf Null controllability of one-dimensional quasilinear parabolic equations via multiplicative controls
\thanks{
This research was supported by National Key Research and Development Program of China under grant 2023YFA1009002, the National Natural Science Foundation of China (12371205, 12371444, 12171166), Jiangxi Provincial Natural Science Foundation (20243BCE51015) and the Natural Science Foundation of Guangdong Province (2025A1515012026).
\newline\indent\hskip2mm{$^\dag$Corresponding author.
E-mail address$:$ mays538@nenu.edu.cn(Y. Ma)}}}
\author{Jilei Huang$^\dag$, Peidong Lei$^\dag$, Yansheng Ma$^\ddag$ and Jingxue Yin$^\dag$
\\[2mm]
\small $^\dag$School of Mathematical Sciences, South China Normal University, Guangzhou, 510631, China
\\
\small $^\ddag$School of Mathematics and Statistics, Northeast Normal University, Changchun, 130024, China}

\date{}
\maketitle

\begin{abstract}
This paper is concerned with the null controllability problem for a class of quasilinear parabolic equations under multiplicative control, locally supported in space.
For the purpose of proving the existence of a multiplicative control forcing the solution rest at a time $T>0$,
we need to establish the decay property of solutions for the system without control first. We have obtained decay estimates for the
$L^\infty$-norm and the $H^1$-norm of solutions to the homogenous quasilinear parabolic equations. Notably, the decay of the $L^\infty$-norm requires no smallness condition
on the initial data, whereas the decay of the $H^1$-norm requires that the $L^\infty$-norm remains small. Based on the decay estimates and maximum modulus estimate of solutions
to quasilinear parabolic equations, together with the local null controllability of quasilinear parabolic equations under additive controls, we prove the null controllability of
the quasilinear parabolic equations via multiplicative controls. As a byproduct, we also obtain the global null controllability for large time to
the quasilinear parabolic equations via additive controls. Given that the controllability under multiplicative control is achieved over a long time horizon,
we finally investigate the existence of time optimal control.

\end{abstract}

\medskip


\medskip


\section{Introduction and main results}
In this paper we shall deal with the exact null controllability of a 1D quasilinear parabolic equation of the form
\begin{equation}\label{1.1}
\left\{\begin{array}{ll}
y_{t}-(a(y)y_x)_x=1_\omega u(g(y)-\theta), &(x,t)\in Q_T,
\\[2mm]
y(0,t)=y(1,t)=0, &t\in(0,T),
\\[2mm]
y(x,0)=y_0(x), &x\in(0,1),
\end{array}\right.
\end{equation}
where $Q_T=(0,1)\times(0, T)$, $1_\omega$ is the characteristic function of a nonempty open subset $\omega$ of $(0,1)$, $u$ is a control input which affects the reaction rate of the process
described by (\ref{1.1}), $y$ is the associated state, $y_0\in H^1_0(0,1)$, $\theta\in L^\infty((0,1)\times {\mathbb R}^+)$ and $g\in C({\mathbb R})$ are given functions.

Quasi-linear parabolic equations, as an important class of diffusion equations, come from a variety of diffusion phenomena in nature (see monographs \cite{V} and \cite{WZYL} and
the rich references therein). It should be noted that the nonlinearity inherent in the principal operator poses new challenges with regard to the study of its controllability,
which are essentially different from the semilinear case. Throughout this paper, the nonlinear diffusion coefficient $a(\cdot): \mathbb{R}\rightarrow \mathbb{R}$ in (\ref{1.1}) is of class $C^2(\mathbb R)$, and
\begin{equation}\label{1.2}
|a'(s)|\le M, \ \ |a''(s)|\le M\qquad \forall s\in \mathbb{R}
\end{equation}
for some $M>0$. In particular, the uniform parabolicity of the operator is guaranteed by the condition
\begin{equation}
\label{1.3}
0<\rho\le a(s)\le \kappa \qquad \forall s\in \mathbb R
\end{equation}
for some $\rho>0$ and $\kappa>0$.

In the past forty years, there have been many advances in the controllability of linear and semilinear parabolic systems (cf. \cite{DFGZ, DZZ, FPZ, FR, FI, IY, LR, LW} and the references therein).
There are some works addressing the controllability problems of quasi-linear parabolic equations. In \cite{Be}, by the fixed point technique, the author has proved the local null controllability
of the following nonlinear diffusion equation in one spacial dimension
\begin{equation}\label{1.4}
\left\{\begin{array}{ll}
y_{t}-(A(y))_{xx}=1_\omega u, &(x,t)\in Q_T,
\\[2mm]
y(0,t)=y(1,t)=0, &t\in(0,T),
\\[2mm]
y(x,0)=y_0(x), &x\in(0,1),
\end{array}\right.
\end{equation}
where $y_0\in H^1_0(0,1)$, $A(\cdot)\in C^3(\mathbb R)$, $A'(s)\ge \rho$ and $\sup_{s\in \mathbb R}\{|A'(s)|, |A''(s)|, |A'''(s)|\}<\infty$.

Here, we say that (\ref{1.4}) is locally null-controllable at time $T>0$ if there exists $\eta>0$ (usually small) such that, for any $y_0\in H^1_0(0,1)$ with $\|y_0\|_{H^1(0,1)}\le\eta$,
there exists a control function $u\in L^2(Q_T)$ such that the associated state $y$ of (\ref{1.4}) satisfies $y(x,T)=0$ in $(0,1)$.

Since then, over the past decade or so, a growing body of research on the controllability of quasilinear parabolic equations has emerged, drawing increasing attention from the researchers.
In \cite{FNNV}, using a local inversion technique (more precisely, Liusternik's inverse function theorem), the authors proved the local null controllability for the internal and boundary control of the following one-dimensional nonlinear diffusion equation
\begin{equation}\label{1.5}
\left\{\begin{array}{ll}
y_{t}-(a(y)y_x)_{x}=1_\omega u, &(x,t)\in Q_T,
\\[2mm]
y(0,t)=y(1,t)=0, &t\in(0,T),
\\[2mm]
y(x,0)=y_0(x), &x\in(0,1),
\end{array}\right.
\end{equation}
and
\begin{equation}\label{1.6}
\left\{\begin{array}{ll}
y_{t}-(a(y)y_x)_{x}=0, &(x,t)\in Q_T,
\\[2mm]
y(0,t)=v(t), \ y(1,t)=0, &t\in(0,T),
\\[2mm]
y(x,0)=y_0(x), &x\in(0,1),
\end{array}\right.
\end{equation}
where $u$ and $v$ are the control functions, $a(\cdot)\in C^2(\mathbb R)$, possesses bounded derivatives of order $\le 2$ and satisfies the condition (\ref{1.3}),
the initial state $y_0\in H^1_0(0,1)$ and $\|y_0\|_{H^1(0,1)}\le\eta$ for a small $\eta>0$.

In \cite{LZ}, the local null controllability for a class of multidimensional quasilinear
parabolic equations have been investigated with homogenous Dirichlet boundary conditions and an arbitrarily located internal controller of the form
\begin{equation}\label{1.7}
\left\{\begin{array}{ll}
y_{t}-\displaystyle\sum_{i,j=1}^n(a^{ij}(y)y_{x_i})_{x_j}+f(y)=1_\omega u, &(x,t)\in \Omega\times (0,T),
\\[2mm]
y(x,t)=0, &(x,t)\in \partial\Omega\times(0,T),
\\[2mm]
y(x,0)=y_0(x), &x\in \Omega,
\end{array}\right.
\end{equation}
where $\Omega$ is a bounded domain in $\mathbb R^n$ with $C^3$ boundary $\partial\Omega$, $\omega$ is a given nonempty open subset of $\Omega$
such that $\overline{\omega}\subseteq\Omega$, $f$ is a given $C^2$ function defined on $\mathbb R$, $f(0)=0$, $a^{ij}(\cdot): \mathbb R\rightarrow \mathbb R$
are twice continuously differentiable functions satisfying $a^{ij}=a^{ji}$ $(i, j= 1, \cdots, n)$, and for some constant $\rho>0$,
$$
\displaystyle\sum_{i,j=1}^n a^{ij}(s)\xi_i\xi_j\ge \rho|\xi|^2 \qquad \forall (s, \xi)= (s, \xi_1, \cdots, \xi_n) \in \mathbb R\times{\mathbb R}^n.
$$
Unlike the local null controllability results in one spacial dimension case in \cite{Be, FNNV}, the null controllability is analyzed in
the frame of classical solutions in \cite{LZ}. The authors found a control function in the H\"{o}lder space $C^{\gamma, \gamma/2}(\overline{\Omega}\times[0,T])$
for the initial state $y_0\in C^{2+\gamma}(\overline{\Omega})$ for any $\gamma\in (0,1)$, with $\|y_0\|_{C^{2+\gamma}(\overline{\Omega})}\le \eta$ for a small $\eta>0$.
The key point of the proof is to establish an observability inequality
for linear parabolic equations, providing an explicit estimate of the observability constant in terms of the $C^1$ coefficients of the principal part.
This is achieved by means of a new global Carleman estimate for general linear parabolic equations.

In \cite{FLM}, the authors consider theoretical and numerical local null controllability result of a quasi-linear parabolic equation in dimensions 2 and 3 in the form of
\begin{equation}\label{1.8}
\left\{\begin{array}{ll}
y_{t}-\mbox{div}(a(y)\nabla y)=\tilde{1}_\omega u, &(x,t)\in \Omega\times (0,T),
\\[1mm]
y(x,t)=0, &(x,t)\in \partial\Omega\times(0,T),
\\[1mm]
y(x,0)=y_0(x), &x\in \Omega,
\end{array}\right.
\end{equation}
where $\Omega\subset \mathbb R^n$ ($n=2,3)$ is a bounded connected open set with a smooth boundary, $\bar{\omega}\subseteq\Omega$ is a nonempty open set,
$\tilde{1}_\omega\in C_0^\infty(\Omega)$ satisfies $0<\tilde{1}_\omega\le 1$ in $\omega$ and $\tilde{1}_\omega=0$ outside $\omega$, $a(\cdot)\in C^3(\mathbb R)$,
possesses bounded derivatives of order $\le 3$ and satisfies the uniform parabolicity condition. They prove that if the initial state $y_0\in H^3(\Omega)\cap H^1_0(\Omega)$ and $\|y_0\|_{H^3(\Omega)}\le\eta$
for a small $\eta>0$, then the system (\ref{1.8}) is locally null-controllable at time $T$.

Recently, in the reference \cite{MCRL}, the authors studied the following quasilinear parabolic
equation with a semilinear term in dimensions 2 and 3:
\begin{equation}\label{1.9}
\left\{\begin{array}{ll}
y_{t}-\mbox{div}(a(y)\nabla y)+f(y)=\tilde{1}_\omega u, &(x,t)\in \Omega\times (0,T),
\\[1mm]
y(x,t)=0, &(x,t)\in \partial\Omega\times(0,T),
\\[1mm]
y(x,0)=y_0(x), &x\in \Omega.
\end{array}\right.
\end{equation}
This paper presents three main results: the local null controllability with distributed controls, the decay of solutions in the $H^3(\Omega)$ space, and the null controllability for large time horizons.
For any $T>0$, the local null controllability is proved for initial states $y_0$ belong to the space $W^{1,p}_0(\Omega)$ with $3<p\le 6$, under the smallness condition $\|y_0\|_{W^{1,p}(\Omega)}\le \eta$
for a small $\eta>0$. This improves the work given in \cite{FLM} since $H^1_0(\Omega)\cap H^3(\Omega)$$\subset$$W^{1,p}_0(\Omega)$ for $3<p\le 6$. The proof relies on the inverse function
theorem and is complemented by maximum regularity results for linear parabolic equations in $L^q(0, T; L^p(\Omega))$ spaces. Moreover, the authors establish the decay of solutions in the $H^3(\Omega)$-norm when the initial
states belong to $H^1_0(\Omega)\cap H^3(\Omega)$, and the $\rho$-null controllability for large time with the initial states in $H^3(\Omega)$.

Some efforts concerning the null controllability of quasi-linear parabolic equations with diffusion coefficients depending on the gradient of the state variable have appeared
in \cite{CLMT} and \cite{FLTM}. These two papers appear to be among the earliest published works on the null controllability of quasi-linear parabolic PDEs with diffusion coefficients that
can take the form of a general power law. In addition, the authors of \cite{WLL} study the local null controllability issue for a free-boundary problem associated with the quasilinear parabolic system (\ref{1.4}).

On the other hand, as is shown in the book of Khapalov (see \cite{K2}, pp.1--5), the classical additive controls are not suitable to deal with a vast array of processes
that can change their principal intrinsic properties due to the control actions. Important examples here include "smart materials" and numerous biomedical, chemical
and nuclear chain reactions, etc. This explains the growing interest in multiplicative controllability.

Note that (\ref{1.1}) is a multiplicative (or bilinear) control system. In the context of heat-transfer the term $u(x, t)(y(x, t)-\theta(x, t))$ is used to describe the heat exchange
at point $(x, t)$ of the given substance with the surrounding medium of temperature $\theta$ according to Newton's Law (see \cite{TS}, pp. 155--156).
Multiplicative systems form a special class of nonlinear systems. They are linear with respect to the input alone and linear with respect to the state alone,
but are not jointly linear in the state and input. This holds true even when the diffusion term itself is linear, as in the case where $a(s)=1$ and $f(s)=s$ in (\ref{1.1}).
As a matter of fact, the dependence of the state with respect to the input is highly nonlinear. Methodologically, the linear duality pairing technique between the control-to-state
mapping at hand and its dual observation map, which is a classical approach to deal with the controllability of additive locally distributed control systems, does not work
for multiplicative control systems, in general. These lead to many difficulties in the study of multiplicative control systems. Due to the negative result given in \cite{BMS},
multiplicative systems have been considered as non-controllable for a long time. However, progress has been made in the last decade and multiplicative systems are now understood in a better way.
See for instance, and the bibliography therein, \cite{ACU},\cite{CK}, \cite{CFK}, \cite{FNT}, \cite{LLG}, \cite{OTB}, and the monograph \cite{K2}. Nevertheless, the aforementioned
works primarily focus on linear or semilinear parabolic equations. As far as we are aware, there is still a lack of literature
concerning the controllability of quasilinear parabolic systems under multiplicative control.

Given that the controllability under multiplicative control is achieved over a long time horizon, we also investigate the existence of time optimal control. The time optimal control problem
was studied first for the finite-dimensional case (cf. \cite{La}). Whereafter, the problem was developed to infinite-dimensional controlled systems (cf. \cite{Fa, LY}). In \cite{Ba},
the time optimal control problem for some controlled parabolic variational inequalities was investigated. However, the method used in \cite{Ba} is suitable only for the case where
the control is distributed in the whole domain $\Omega$. In \cite{W}, the time optimal control was obtained for the equation
\begin{align}
\label{1.10} y_{t}-\Delta
y+f(y)=1_\omega u\quad \mbox{in}\ Q_\infty,
\end{align}
where the control $u$ acts only on a local domain $\omega$, and the aim function is the steady-state solution $y_e$ to (\ref{1.10}),
i.e., $-\Delta y_e(x)+f(y_e(x))=0$ in $\Omega$ and $y_e(x)=0$ on $\partial\Omega$.

As is shown in \cite{W}, the key to get the existence of a time optimal control is to show the existence of an admissible control
which is related a type of controllability of the equation with some kind of control constraint. However, the existing research on time-optimal control for parabolic equations
that we have seen is primarily focused on linear and semilinear parabolic equations (cf. \cite{LZH}, \cite{WZ}, \cite{WuLW}, and the monograph \cite{WWXZ}).
Results on time optimal control problems for quasilinear parabolic equations are still scarce. Therefore, in this paper, we investigate the existence of time optimal
control for the quasilinear multiplicative control system (\ref{1.1}).

\vskip1mm

The main results of this paper are stated as follows.
\begin{theorem}
Let $g\in C(\mathbb{R})$, $g(0)=0$ and $g(s)$ be Lipschitz
continuous on $[-L, L]$ for any $L>0$. Assume that $\theta\in L^{\infty}((0,1)\times \mathbb R^+)$ satisfies
\begin{equation}\label{1.11}
|\theta(x,t)|\geq \theta_0>0\quad a.e. \ \ \mbox{in}\ \omega\times(0,+\infty)
\end{equation}
for some $\theta_0>0$. If the nonlinear diffusion coefficient $a(s)$ satisfies the conditions (\ref{1.2}) and (\ref{1.3}), then for any initial state $y_0\in H^1_0(0,1)$,
there exist $T=T(y_0, \theta)>0$ and a multiplicative control $u\in L^{\infty}(Q_T)\cap H^1([0,T]; L^2(0,1))$ such that the corresponding solution $y$ of (\ref{1.1}) satisfies
\begin{equation*}
y(x,T)=0 \quad \mbox{in} \ (0,1).
\end{equation*}
\end{theorem}

\begin{remark}\label{1.1}
It is worth pointing out that there is no restriction
on the growth of the nonlinearity $g(s)$ with respect to the variable $s$ in our proofs.
\end{remark}

\begin{theorem}
If the nonlinear diffusion coefficient $a(s)$ satisfies the conditions (\ref{1.2}) and (\ref{1.3}), then for any initial state $y_0\in H^1_0(0,1)$,
there exist $T>0$ and a control $u\in L^{\infty}(Q_T)\cap H^1([0,T]; L^2(0,1))$ such that the corresponding solution $y$ of (\ref{1.5}) satisfies
\begin{equation*}
y(x,T)=0 \quad \mbox{in} \ (0,1).
\end{equation*}\end{theorem}

\begin{remark}\label{1.2}
The decay estimates of solutions play a crucial role in the proofs of Theorems 1.1 and 1.2. Specifically, we have obtained the decay estimate for the
$L^\infty$-norm of solutions to the quasilinear parabolic equation (2.1) (see Proposition 2.1), as well as the decay estimate for the $H^1$-norm (see Proposition 2.2).
Notably, the decay of the $L^\infty$-norm does not require any smallness condition on the initial data, whereas the decay estimate for the $H^1$-norm is established
under the condition that the $L^\infty$-norm remains small.
\end{remark}

\begin{remark}\label{1.3}
In \cite{MCRL}, the authors also established the large-time null controllability for system (\ref{1.5}) in dimensions 2 and 3. They utilized the decay property of the $H^3$-norm of the solution.
However, to obtain this decay property, it is required that the initial datum $y_0$ belongs to $H^3(\Omega)\cap H^1_0(\Omega)$, and its $H^3(\Omega)$-norm satisfies a smallness condition.
\end{remark}

Let $u$ be a control taken from a given set
\begin{equation}
\label{1.12}
U_{\sigma}=\{u; \ u\in L^\infty(Q_\infty),
|u|\le\sigma \ \mbox{a.e. in}\ Q_\infty\},
\end{equation}
where $\sigma>0$ is an arbitrary but fixed positive constant. Our third result is concerned with the following time optimal control problem:
$$
\mbox{(P)}\quad\displaystyle\min\{T; \ y(\cdot, T)=0\ \mbox{a.e. in}\ (0,1),
\ y \ \mbox{is the solution to}\ \eqref{1.1} \ \mbox{corresponding to}\ u\in U_\sigma\}.\qquad\qquad
$$

A function $u\in U_\sigma$ is called admissible if the corresponding solution $y$ to
(\ref{1.1}) satisfying $y(\cdot,T)=0$ a.e. in $(0,1)$ for some $T>0$. Define
\begin{equation}\label{1.13}
T^*(\sigma)\triangleq\min\{T;\ y(\cdot,
T)=0\ \mbox{a.e. in}\ (0,1), u\in U_\sigma\}.
\end{equation}
$T^*$ is called the minimal time for the problem (P), and a control $u^*\in U_\sigma$ such that the corresponding state $y^*(\cdot, T^*(\sigma))=0$ a.e. in $(0,1)$ is
called a time optimal control.

\begin{theorem} If $y_0$ and $g(\cdot)$ satisfy the assumptions in Theorem 1.1, then for any $\sigma>0$,
there exists at least one time optimal control for the problem (P).
\end{theorem}

\begin{remark}
In Theorems 1.1 and 1.3, the nonlinear function $g(s)\in C(\mathbb R)$ depends only on the variable $s$. Suppose that $g\in C([0,1]\times[0,+\infty)\times\mathbb R)$ satisfies
$g(\cdot, \cdot, 0)=0$. Moreover, assume that for any $L>0$, the function $g(x,t, \cdot)$ is Lipschitz continuous on the interval $[-L, L]$, uniformly in $[0,1]\times[0,+\infty)$.
If, in addition, there exists a function $G\in C^1(\mathbb R)$ such that $|g(x,t, s)|\le G(s)$ for all $(x,t)\in [0,1]\times[0,+\infty)$ and $s\in \mathbb R$, then Theorems 1.1 and 1.3 remain valid.
\end{remark}

The rest of this paper is organized as follows. Section 2 is devoted to establishing decay estimates for the $L^\infty$-norm and the $H^1$-norm of solutions to
the homogeneous quasilinear parabolic equations, as well as the maximum modulus estimate for solutions to the corresponding inhomogeneous equations. In Section 3, we give an $L^\infty$-cost estimate
for the control function when the additive control system (\ref{1.5}) achieves null controllability. In Section 4,
we prove Theorems 1.1 and 1.2. Finally, in Section 5, we prove, without imposing any restrictions on the growth of the nonlinearity $g(s)$ with respect to the variable $s$, the existence of a local solution with duration 1
in time for equation (\ref{1.1}), as well as the existence of a time optimal control.

\section{Estimates for solutions of quasilinear parabolic equations}
Consider the following quasilinear equation
\begin{equation}\label{2.1}
\left\{\begin{array}{ll}
y_t-(a(y)y_x)_x=0, &(x,t)\in Q_T,
\\[2mm]
y(0,t)=y(1,t)=0, &t\in (0,T),
\\[2mm]
y(x,0)= y_0(x), &x \in (0,1),
\end{array}\right.
\end{equation}
where the diffusion coefficient $a(s)$ is assumed to satisfy conditions (\ref{1.2}) and (\ref{1.3}).

According to the classical theory of parabolic equations, under conditions (\ref{1.2}) and (\ref{1.3}), if the initial state $y_0$ is in $H^1_0(0,1)$, then the problem (\ref{2.1}) admits a unique solution $y\in\stackrel{\circ}{W}{\hskip-1mm}^{2,1}_2(Q_T)$. Here $\stackrel{\circ}{W}{\hskip-1mm}^{2,1}_2(Q_T)$ denotes the closure in $W^{2,1}_2(Q_T)$ of all functions infinitely differentiable on $\overline{Q}_T$,
which vanishes near the lateral boundary $\{0,1\}\times(0,T)$.

Since we are in the one-dimensional setting, the embedding theorem yields
\begin{equation}\label{2.2}
W^{2,1}_2(Q_T)\subset C^{\alpha, \alpha/2}(\overline{Q}_T), \quad \ \alpha\in (0, 1/2].
\end{equation}

For any $T>0$, denote
$$
\varGamma_{T}=[0,1]\times[0,T]\backslash(0,1)\times(0,T].
$$
It is commonly referred to as the parabolic boundary of $Q_T$.

For $w\in W_2^{1,1}(Q_T)$, define the least upper bound of $w$ on $Q_T$ and $\varGamma_T$ as
$$
\sup_{Q_T}w=\inf\{l;\ (w-l)_+=0,\ \text{a.e. in } Q_T\},
$$
and
$$
\sup_{\varGamma_T}w=\inf\left\{l;\ (w-l)_+ \in\stackrel{\bullet}{W}{\hskip-1mm}^{1,1}_2(Q_T)\right\},
$$
where $s_+=\max\{s,0\}$, $\stackrel{\bullet}{W}{\hskip-1mm}^{1,1}_2(Q_T)$ denotes the closure in $W^{1,1}_2(Q_T)$ of all functions infinitely differentiable on $\overline{Q}_T$, which vanish near the parabolic boundary $\varGamma_T$.

\begin{lemma}\label{Lemma 2.1}
If $y\in \stackrel{\circ}{W}{\hskip-1mm}^{1,1}_2(Q_T)$ is a weak solution of the homogeneous quasilinear parabolic problem (\ref{2.1}), then
\begin{equation}\label{2.3}
\sup_{\overline{Q}_T}y=\sup_{\varGamma_T}y, \ \ \inf_{\overline{Q}_T}y=\inf_{\varGamma_T}y.
\end{equation}
\end{lemma}
{\bf Proof.} \ For any $\varphi\in\stackrel{\circ}{W}{\hskip-1mm}^{1,1}_2(Q_T)$, $y$ satisfies
\begin{equation}\label{2.4}
\iint_{Q_T}\left(y_t\varphi+a(y)y_x\varphi_x\right)dxdt=0
\end{equation}

Set $l=\sup_{\varGamma_T} y$ and choose $\varphi=(y-k)_+$ with $k>l$ in (\ref{2.4}). Then $\varphi\in \stackrel{\bullet}{W}{\hskip-1mm}^{1,1}_2(Q_T)\subset \stackrel{\circ}{W}{\hskip-1mm}^{1,1}_2(Q_T)$. Substituting $\varphi$ into (\ref{2.4}) and using the operation rules of weak derivatives, give
\[
\iint_{Q_T}(y-k)_t(y-k)_+dxdt+\iint_{Q_T}a(y)(y-k)_x{(y-k)_+}_xdxdt=0,
\]
i.e.,
\[
\frac{1}{2}\iint_{Q_T}{(y-k)_+^2}_tdxdt+\iint_{Q_T}a(y)\big|{(y-k)_+}_x\big|^2dxdt=0.
\]
Thus
\[
\frac{1}{2}\int_0^1(y(x,T)-k)_+^2dx-\frac{1}{2}\int_0^1(y(x,0)-k)_+^2dx+\iint_{Q_T}a(y)|{(y-k)_+}_x|^2dxdt=0.
\]
Since $k>l$, we have $\int_0^1(y(x,0)-k)_+^2dx=0$, and hence
$$
\iint_{Q_T}a(y)\big|{(y-k)_+}_x\big|^2dxdt\leq 0.
$$
Combining this with Poincar$\acute{\mbox{e}}$'s inequality in the space $H^1_0(0,1)$ and the uniformly parabolicity condition (\ref{1.3}), we further obtain
\[
\iint_{Q_T}(y-k)_+^2dxdt\leq\frac{1}{\pi^2}\iint_{Q_T}\big|{(y-k)_+}_x\big|^2dxdt\leq\frac{1}{\pi^2\rho}\iint_{Q_T}a(y)\big|{(y-k)_+}_x\big|^2dxdt\le 0.
\]
Hence $y(x,t)\leq k$ a.e. in $Q_T$ and the first equality in the lemma follows from the arbitrariness of $k>l$ and (\ref{2.2}). The second equality
$\inf_{\overline{Q}_T}y=\inf_{\varGamma_T}y$ follows by a similar argument. The proof is complete. $\Box$

\vskip2mm
Let $y\in \stackrel{\circ}{W}{\hskip-1mm}^{2,1}_2(Q_T)$ be a solution of the homogeneous quasilinear parabolic problem (\ref{2.1}). From (\ref{2.2}), we see that
$$
\max_{\overline{Q}_T}y=\sup_{\overline{Q}_T}y, \ \min_{\overline{Q}_T}y=\inf_{\overline{Q}_T}y, \ \max_{\varGamma_T}y=\sup_{\varGamma_T}y, \ \min_{\varGamma_T}y=\inf_{\varGamma_T}y.
$$
In view of Lemma 2.1,
$$
\max_{\overline{Q}_T}|y|=\max\{\max_{\overline{Q}_T}y, -\min_{\overline{Q}_T}y\}=\max\{\max_{\varGamma_T}y, -\min_{\varGamma_T}y\}=\max_{\varGamma_T}|y|,
$$
and we deduce that
\begin{corollary}\label{Coro.2.1}
Let $y\in\stackrel{\circ}{W}{\hskip-1mm}^{2,1}_2(Q_T)$ be a solution of the homogeneous quasilinear parabolic problem (\ref{2.1}). Then
\begin{equation}\label{2.5}
\max_{\overline{Q}_T}|y|=\max_{\varGamma_T}|y|.
\end{equation}
\end{corollary}

\begin{proposition}\label{Prop2.1}
Let $y_0\in H^1_0(0,1)$ and let $y\in\stackrel{\circ}{W}{\hskip-1mm}^{2,1}_2(Q_T)$ be a solution of the homogeneous quasilinear parabolic problem (\ref{2.1}). Then
\begin{equation}\label{2.6}
\max\limits_{x\in[0, 1]}|y(x, t)|\leq (2\rho)^{-1/2}\|y_0\|_{L^2(0,1)}t^{-1/2}.
\end{equation}
\end{proposition}
{\bf Proof.} \ For any $t\in [0, T]$, it follows from (\ref{2.2}) that $y(\cdot, t)\in C^{\alpha}([0,1])$. Therefore, we can define
$$
M(t)=\max\limits_{x\in[0, 1]}|y(x, t)|.
$$

We assert further that
\begin{equation}\label{2.7}
M(t_{1})\geq M(t_{2}), \quad \forall 0\le t_{1}<t_{2}\le T.
\end{equation}
In fact, if there exist $t_{1},\ t_{2}\in (0, T)$ with $t_{1}<t_{2}$, such that $M(t_{1})< M(t_{2})$, then there exist $x_1\in [0,1]$ and $x_2\in [0,1]$, respectively, such that
$$
|y(x_{2}, t_{2})|=\max\limits_{x\in[0, 1]}|y(x, t_{2})|>\max\limits_{x\in[0, 1]}|y(x, t_{1})|=|y(x_{1},t_{1})|.
$$
Note that $y(0,t)=y(1,t)=0$ implies
$$
\max\limits_{x\in[0, 1]}|y(x, t_{1})|=\max\limits_{(x, t)\in\varGamma^{t_{2}}_{t_{1}}}|y(x, t)|,
$$
where $\varGamma_{t_1}^{t_2}=[0,1]\times[t_1,t_2]\backslash(0,1)\times(t_1,t_2]$. Therefore,
$$
\begin{aligned}
\max\limits_{(x, t)\in [0,1]\times[t_1,t_2]}|y|\geq |y(x_{2}, t_{2})|=\max\limits_{x\in[0, 1]}|y(x, t_{2})|>\max\limits_{(x, t)\in\varGamma^{t_{2}}_{t_{1}}}|y|.
\end{aligned}
$$
This contradicts (\ref{2.5}). Thus, we know that $M(t)$ is monotonically non-increasing.

Utilizing the Newton-Leibniz formula and H\"{o}ler's inequality, we deduce that
\begin{equation}\label{2.8}
\begin{aligned}
|y(x, t)|=&\left|\int^{x}_{0}y_{z}(z, t)dz\right|\leq\sqrt{x}\left(\int^{x}_{0}|y_{z}(z, t)|^{2}dz\right)^{1/2}
\end{aligned}
\end{equation}
for $x\in(0, 1)$, and hence
\begin{equation}\label{2.9}
M^{2}(t)\leq \int^{1}_{0}|y_{z}(z, t)|^{2}dz.
\end{equation}
Multiplying \eqref{2.1} by $y$, integrating over $(0, 1)$ with respect to $x$ and recalling \eqref{1.3}, we obtain
\begin{equation}\label{2.10}
\begin{aligned}
\frac{1}{2}\frac{d}{dt}\int^{1}_{0}y^{2}(x, t)dx=-\int^{1}_{0}a(y(x,t))y^{2}_{x}(x,t)dx\leq-\rho\int^{1}_{0}y^{2}_{x}(x,t)dx.
\end{aligned}
\end{equation}
Combining (\ref{2.9}) and (\ref{2.10}), we discover
\begin{equation}\label{2.11}
\frac{d}{dt}\int^{1}_{0}y^{2}(x, t)dx\leq-2\rho M^{2}(t).
\end{equation}
Now we integrate \eqref{2.11} over $(0,t)$ and employ the monotonic non-increasing property of $M(t)$ to deduce
\begin{equation}\label{2.12}
\int^{1}_{0}y^{2}(x, t)dx-\int^{1}_{0}y^{2}_0dx\leq -2\rho\int^{t}_{0}M^{2}(\tau)d\tau
\leq-2\rho\int^{t}_{0}M^{2}(t)d\tau=-2\rho M^{2}(t)t.
\end{equation}
Thus
$$
M^{2}(t)\leq\frac{\int^{1}_{0}y^{2}_{0}dx}{2\rho}t^{-1}.
$$
Consequently this implies \eqref{2.6}. The proof is complete. \hfill $\Box$
\vskip2mm
Before proving the $H^1$-decay estimate for the solution of the (\ref{2.1}), let us first recall the Gagliardo-Nirenberg interpolation inequality (see \cite{N}).

For any $u\in L^q(0,1)$ and $D^m u\in L^r(0,1)$, we have
\begin{equation}\nonumber
\|D^k u\|_{L^p(0,1)} \leq C \|u\|_{L^q(0,1)}^{1-\alpha}\|D^m u\|_{L^r(0,1)}^{\alpha}+C\|u\|_{L^l(0,1)},
\end{equation}
where
$$
1\leq q,r\le+\infty, \ 0\le k<m, \ \frac{1}{p}=k+\left(\frac{1}{r}-m\right)\alpha+\frac{1-\alpha}{q}, \ \quad \forall \alpha \in \left[\frac{k}{m}, 1\right],\ l>1,
$$
and $C$ is a constant depending only on $q, r, k, m,\alpha$ and $l$, and does not depend on $u$. Take $k=1$, $p=4$, $m=2$, $r=2$, $q=+\infty$, $\alpha=1/2$, $l=2$, we discover
\begin{align}
\label{2.13}
\|u_x\|_{L^4(0,1)}\leq C_0\|u\|_{L^\infty(0,1)}^{1/2}\|u_{xx}\|_{L^2(0,1)}^{1/2}+C_0\|u\|_{L^2(0,1)},
\end{align}
where $C_0$ is a constant independent of $u$.

From $u\in H^2(0,1)$, it immediately follows that $u\in L^\infty(0,1)$ and $u_x\in L^4(0,1)$. Here, we fully exploit the advantage brought by the one-dimensional nature of the space.

\begin{proposition}\label{Prop2.2}
Let $y_0\in H^1_0(0,1)$ and let $y\in \stackrel{\circ}{W}{\hskip-1mm}^{2,1}_2(Q_T)$ be a solution of the homogeneous quasilinear parabolic problem (\ref{2.1}). If
\begin{equation}
\label{2.14}
\|y(\cdot,t)\|_{L^\infty(0,1)}\le \frac{\rho}{4MC_0^2},
\end{equation}
then
\begin{equation}\label{2.15}
\|y(\cdot, t)\|_{H^1(0,1)}\leq \exp\Big\{-\frac{(2\pi^2-1)\rho t}{2(\pi^2+1)}\Big\}\|y_0\|_{H^1(0,1)}.
\end{equation}
\end{proposition}
{\bf Proof.} \ Define the energy
\begin{equation}
\label{2.16}
E(t):=\frac{1}{2}\int_0^1\left( y^2(x,t) + y_x^2(x,t)\right) dx.
\end{equation}
Thus, differentiating (\ref{2.16}) with respect to $t$, we find
\begin{equation}
\label{2.17}
E'(t)=\int_0^1 \left( y y_t + y_x y_{xt} \right)dx:=J_1(t)+J_2(t).
\end{equation}

Next, we proceed to handle $J_1$ and $J_2$ separately.

Substituting $y_t = (a(y)y_x)_x$, we get by integrating by parts that
\begin{equation}
\label{2.18}
J_1(t)=\int_0^1 yy_tdx=\int_0^1 y (a(y)y_x)_x dx=ya(y)y_x\big|_0^1-\int_0^1 a(y) y_x^2dx=-\int_0^1 a(y) y_x^2dx.
\end{equation}
In view of the uniform parabolicity condition (\ref{1.3}), we obtain
\begin{equation}
\label{2.19}
J_1(t)=-\int_0^1 a(y) y_x^2 dx\leq-\rho\int_0^1 y_x^2 dx.
\end{equation}

For the second term $J_2$, by using the homogeneous boundary condition, we have $y_t(0, t)=y_t(1, t)=0$, and thus
\begin{align}
\begin{aligned}
J_2(t)&=y_x y_t\Big|_0^1-\int_0^1 y_{xx} y_tdx
\\
&=-\int_0^1 y_{xx}(a(y)y_x)_xdx
\\
&=-\int_0^1 y_{xx}\left(a'(y)y_x^2 + a(y)y_{xx}\right)dx
\\
&=-\int_0^1 a'(y)y_x^2 y_{xx}dx-\int_0^1 a(y)y_{xx}^2dx.
\end{aligned}
\label{2.20}
\end{align}
The condition (\ref{1.2}) and Cauchy's inequality with $\varepsilon$ imply
\begin{equation}
\left|-\int_{0}^{1} a'(y) y_x^2 y_{xx}dx\right|\leq M \int_{0}^{1} |y_x|^2 |y_{xx}|dx
\leq \frac{M}{2\varepsilon}\int_{0}^{1}y_x^4dx+\frac{M\varepsilon}{2}\int_{0}^{1} y_{xx}^2dx.
\label{2.21}
\end{equation}
By (\ref{2.13}), we have
\begin{align}
\nonumber
\|y_x(\cdot, t)\|_{L^4(0,1)}\leq C_0\|y(\cdot,t)\|_{L^\infty(0,1)}^{1/2}\|y_{xx}(\cdot,t)\|_{L^2(0,1)}^{1/2}+C_0\|y(\cdot,t)\|_{L^2(0,1)}.
\end{align}
Substituting into (\ref{2.21}), we discover
\begin{align}
\begin{aligned}
&\left|-\int_{0}^{1} a'(y) y_x^2 y_{xx}dx\right|
\\[1mm]
\leq&\frac{M}{2\varepsilon}\left[C_0\|y(\cdot,t)\|_{L^\infty(0,1)}^{1/2}\left(\int_{0}^1 y_{xx}^2dx\right)^{1/4}+C_0\left(\int_{0}^1 y^2dx\right)^{1/2}\right]^4+\frac{M\varepsilon}{2}\int_{0}^{1}y_{xx}^2dx
\\[1mm]
\le&\frac{8MC_0^4}{\varepsilon}\left[\|y(\cdot,t)\|_{L^\infty(0,1)}^{2}\int_{0}^1 y_{xx}^2dx+\left(\int_{0}^1 y^2dx\right)^2\right]+\frac{M\varepsilon}{2}\int_{0}^{1}y_{xx}^2dx
\\[1mm]
\le&\frac{8MC_0^4}{\varepsilon}\|y(\cdot,t)\|_{L^\infty(0,1)}^{2}\left(\int_{0}^1 y_{xx}^2dx+\int_{0}^1 y^2dx\right)+\frac{M\varepsilon}{2}\int_{0}^{1}y_{xx}^2dx.
\end{aligned}
\nonumber
\end{align}
Under the assumption (\ref{2.14}) and with $\varepsilon=\rho/M$, we conclude
\begin{align}
\begin{aligned}
J_2(t)&\leq \frac{\rho^2}{2M\varepsilon}\int_{0}^1y_{xx}^2dx+\frac{M\varepsilon}{2}\int_{0}^{1}y_{xx}^2dx-\int_{0}^{1}a(y)y_{xx}^2dx+\frac{\rho^2}{2M\varepsilon}\int_{0}^1 y^2dx
\\[1mm]
&\leq \left(\frac{\rho^2}{2M\varepsilon}+\frac{M\varepsilon}{2}-\rho\right)\int_{0}^1y_{xx}^2dx+\frac{\rho^2}{2M\varepsilon}\int_{0}^1 y^2dx
\\[1mm]
&=\frac{\rho}{2}\int_{0}^1 y^2dx.
\end{aligned}
\label{2.22}
\end{align}
Consequently, it follows from \eqref{2.17}, (\ref{2.19}) and \eqref{2.22} that
\begin{align}
\begin{aligned}
E'(t) \leq -\rho\int_0^1 y_x^2(x,t)dx+\frac{\rho}{2}\int_{0}^1 y^2(x,t)dx.
\end{aligned}
\label{2.23}
\end{align}
Using the Poincar$\acute{\mbox{e}}$ inequality on the space $H^1_0(0,1)$, we have
\begin{equation}
\label{2.24}
\int_{0}^1 y^2(x,t)dx\le \frac{1}{\pi^2}\int_0^1 y_x^2(x,t)dx.
\end{equation}
Combining (\ref{2.23}) and (\ref{2.24}), and splitting $\rho$ into $(2\pi^2-1)\rho/(2\pi^2+2)$ and $3\rho/(2\pi^2+2)$, leads to
\begin{align}
\begin{aligned}
E'(t)&\leq -\frac{(2\pi^2-1)\rho}{2(\pi^2+1)}\int_0^1 y_x^2(x,t)dx-\frac{3\rho}{2(\pi^2+1)}\int_0^1 y_x^2(x,t)dx+
\frac{\rho}{2}\int_{0}^1 y^2(x,t)dx
\\[1mm]
&\le-\frac{(2\pi^2-1)\rho}{2(\pi^2+1)}\int_0^1 y_x^2(x,t)dx-\frac{3\pi^2\rho}{2(\pi^2+1)}\int_0^1 y^2(x,t)dx+
\frac{\rho}{2}\int_{0}^1 y^2(x,t)dx
\\[1mm]
&=-\frac{(2\pi^2-1)\rho}{2(\pi^2+1)}\left(\int_0^1 y^2(x,t)dx+\int_0^1 y_x^2(x,t)dx\right)
\\[1mm]
&=-\frac{(2\pi^2-1)\rho}{\pi^2+1}E(t),
\end{aligned}
\label{2.25}
\end{align}
and hence
\begin{align}
\nonumber
E(t)\leq E(0)\exp\Big\{-\frac{(2\pi^2-1)\rho t}{\pi^2+1}\Big\},
\end{align}
which implies the decay estimate (\ref{2.15}). The proof is complete. \hfill $\Box$
\vskip2mm

In \cite{Be}, an estimate (see (4.41) in \cite{Be}) that will be needed in the subsequent proofs of this paper is established by using the regularizing properties of (\ref{2.1}).
Without presenting its proof, we state it directly as the following proposition:
\begin{proposition}\label{Prop2.3}
Let $y\in \stackrel{\circ}{W}{\hskip-1mm}^{2,1}_2(Q_T)$ be a solution of the homogeneous quasilinear parabolic problem (\ref{2.1}). For any $t_0>0$,
\begin{equation}\label{2.26}
\int_0^1\big[(A(y(x,t_0)))_{xx}^2+y^2_x(x,t_0)+y^2(x,t_0)\big]dx\le C\int_0^1 (y_0^2+(y_0)_x^2)dx
\end{equation}
for any $y_0$ with $\|y_0\|_{H^1(0,1)}$ sufficiently small, where $C$ is a constant depending on $\rho$ and $t_0$.
\end{proposition}

For the subsequent proof of the controllability of the multiplicative control system (\ref{1.1}), we provide the following maximum modulus estimate for the solution of the following problem:
\begin{equation}\label{2.27}
\left\{\begin{array}{ll}
y_t-(a(y)y_x)_x=f, &(x,t)\in Q_T,
\\[2mm]
y(0,t)=y(1,t)=0, &t\in (0,T),
\\[2mm]
y(x,0)=y_0(x), &x \in (0,1),
\end{array}\right.
\end{equation}
where $y_0\in H^1_0(0,1)$, $f\in L^\infty(Q_T)$, and the diffusion coefficient $a(\cdot)$ satisfies the boundedness and the uniform parabolicity conditions given in (\ref{1.2}) and (\ref{1.3}), respectively.

We need the following elementary inequality, which is a direct corollary of a recursion inequality in the monograph \cite{LSU} (see Lemma 5.6, pp. 95), and
its proof can be found in \cite{WYW} (see Lemma 4.1.4, pp. 107).

\begin{lemma}\label{Lemma 2.2}
Let $\mu(h)$ be a nonnegative and nonincreasing function on $[k_0, +\infty)$, satisfying
\begin{equation}\nonumber
\mu(h)\leq\left(\frac{M}{h - k}\right)^\alpha[\mu(k)]^\beta, \quad \forall h>k\geq k_0
\end{equation}
for some constants $M>0$, $\alpha>0$, $\beta>1$. Then there exists $d>0$ such that
\[
\mu(h)=0, \quad \forall h\geq k_0+d.
\]
\end{lemma}

Denote by $V_2(Q_T)$ the set $L^\infty([0,T]; L^2(0,1))\cap W_2^{1,0}(Q_T)$ endowed with the norm
$$
\|y\|_{V_2(Q_T)}=\mbox{ess}\hskip-0.5mm\sup_{0\le t\le T}\|y(\cdot, t)\|_{L^2(0,1)}+\|y_x\|_{L^2(Q_T)}.
$$
Let $\stackrel{\circ}{V}{\hskip-1mm}^{1,0}_2(Q_T)$ denote the normed linear space consisting of those functions $y$ in $V_2(Q_T)$ which satisfy
$$
\lim_{h\to 0}\|y(\cdot, t+h)-y(\cdot,t)\|_{L^2(0,1)}=0,
$$
and $y(0,t)=y(1,t)=0$, for all $t, t+h\in [0,T]$.

According to the $L^2$-theory of the quasilinear parabolic equations (see Theorem 3.3 in \cite{Ch} page 196), the problem (\ref{2.27}) admits a unique
solution $y\in L^\infty(Q_T)\cap\stackrel{\circ}{V}{\hskip-1mm}^{1,0}_2(Q_T)$.

\begin{proposition}\label{Prop2.4}
Let $y_0\in H^1_0(0,1)$, $f\in L^\infty(Q_T)$, and let $y\in L^\infty(Q_T)\cap\stackrel{\circ}{V}{\hskip-1mm}^{1,0}_2(Q_T)$ be a weak solution to the problem (\ref{2.27}). Then
\begin{equation}\label{2.28}
\sup_{Q_T}y\leq\sup_{[0,1]}y_0+\frac{6}{\rho}\|f\|_{L^\infty(Q_T)},
\end{equation}
\begin{equation}\label{2.29}
\inf_{Q_T}y\geq\inf_{[0,1]}y_0-\frac{6}{\rho}\|f\|_{L^\infty(Q_T)}.
\end{equation}
\end{proposition}
{\bf Proof.} \ Denote $\displaystyle\sup_{\varGamma_T} y=l$. For $k>l$ and $0\leq t_1<t_2\leq T$, it is easy to check that $\varphi=(y-k)_+\chi_{[t_1,t_2]}(t)\in \stackrel{\circ}{W}{\hskip-1mm}^{1,0}_2(Q_T)$,
where $\chi_{[t_1,t_2]}(t)$ is the characteristic function of the interval $[t_1,t_2]$. Thus we may choose $\varphi$ as a test function to obtain
\[
\iint_{Q_T}(y-k)_t(y-k)_+\chi_{[t_1,t_2]}dxdt+\iint_{Q_T}\chi_{[t_1,t_2]}a(y)\big|{(y-k)_+}_x\big|^2dxdt=\iint_{Q_T}f(y-k)_+\chi_{[t_1,t_2]}dxdt.
\]
Hence
\[
\frac{1}{2}(I_k(t_2)-I_k(t_1))+\int_{t_1}^{t_2}\hskip-1.5mm\int_0^1a(y)\big|{(y-k)_+}_x\big|^2dxdt\leq\int_{t_1}^{t_2}\hskip-1.5mm\int_0^1|f|(y - k)_+dxdt,
\]
where
\[
I_k(t)=\int_0^1(y(x,t)-k)_+^2dx.
\]

Assume that the absolutely continuous function $I_k(t)$ attains its maximum at $\xi\in[0,T]$. Since $I_k(0)=0$, $I_k(t)\geq 0$, we may suppose $\xi>0$.
Taking $t_1=\xi-\varepsilon$, $t_2=\xi$ with $\varepsilon>0$ small enough so that $\xi-\varepsilon>0$ and noticing that $I_k(\xi)-I_k(\xi-\varepsilon)\geq 0$, we obtain
\begin{equation}
\label{2.30}
\int_{\xi-\varepsilon}^\xi\hskip-0.5mm\int_0^1a(y)\big|{(y-k)_+}_x\big|^2dxdt\leq\int_{\xi-\varepsilon}^\xi\hskip-0.5mm\int_0^1|f|(y-k)_+dxdt,
\end{equation}
and therefore
\[
\frac{1}{\varepsilon}\int_{\xi-\varepsilon}^\xi\hskip-0.5mm\int_0^1\big|{(y-k)_+}_x\big|^2dxdt\leq\frac{1}{\varepsilon\rho}\int_{\xi-\varepsilon}^\xi\hskip-0.5mm\int_0^1|f|(y-k)_+dxdt.
\]
Letting $\varepsilon\to 0^+$, we derive
\[
\int_0^1\big|{(y(x,\xi)-k)_+}_x\big|^2dx\leq\frac{1}{\rho}\int_0^1|f(x,\xi)|(y(x,\xi)-k)_+dx.
\]
For any $p>2$, using the Sobolev embedding theorem gives
\begin{equation}
\label{2.31}
\left(\int_0^1|(y(x,\xi)-k)_+|^pdx\right)^{2/p}\le \frac{1}{\rho}\int_0^1|f(x,\xi)|(y(x,\xi)-k)_+dx.
\end{equation}

Denote
\[
A_k(t)=\{x\in (0,1);\ y(x,t)>k\}, \quad \mu_k=\sup_{0<t<T}|A_k(t)|,
\]
where the symbol $|E|$ represents the Lebesgue measure of the measurable set $E$. An application of H\"{o}lder's inequality to (\ref{2.31}) yields
\begin{equation}\label{2.32}
\left(\int_{A_k(\xi)}(y-k)_+^pdx\right)^{1/p}\leq \frac{1}{\rho}\left(\int_{A_k(\xi)}|f|^q dx\right)^{1/q}\leq \frac{1}{\rho}\|f\|_{L^\infty(Q_T)}|A_k(\xi)|^{1/q}\leq \frac{1}{\rho}\|f\|_{L^\infty(Q_T)}\mu_k^{1/q},
\end{equation}
where $q=\frac{p}{p-1}$.

Utilizing H\"{o}lder's inequality to $I_k(\xi)$ and combining the result with (\ref{2.32}), we are led to
\[
\begin{aligned}
I_k(\xi)\leq\left(\int_{A_k(\xi)}(y(x,\xi)-k)_+^p dx\right)^{2/p}|A_k(\xi)|^{(p-2)/p}\leq\left(\frac{1}{\rho}\|f\|_{L^\infty(Q_T)}\right)^2\mu_k^{(3p-4)/p}.
\end{aligned}
\]
Hence, for any $t\in[0,T]$,
\begin{equation}\label{2.33}
I_k(t)\leq I_k(\xi)\leq(\|f\|_{L^\infty(Q_T)}/\rho)^2\mu_k^{(3p-4)/p}.
\end{equation}
Since for any $h>k$ and $t\in[0,T]$,
\[
I_k(t)\geq\int_{A_h(t)}(y(x,t)-k)_+^2 dx\geq(h-k)^2|A_h(t)|,
\]
from (\ref{2.33}) we obtain
\[
(h-k)^2\mu_h\leq(\|f\|_{L^\infty(Q_T)}/\rho)^2\mu_k^{(3p-4)/p},
\]
i.e.,
\[
\mu_h\leq\left(\frac{\|f\|_{L^\infty(Q_T)}}{\rho(h-k)}\right)^2\mu_k^{(3p-4)/p}.
\]
Using Lemma 2.2 and noticing that $p>2$ implies $(3p-4)/p>1$, we finally arrive at
\[
\mu_{l+d}=\sup_{0<t<T}|A_{l+d}(t)|=0,
\]
where
\[
\begin{aligned}
d&=\|f\|_{L^\infty(Q_T)}\rho^{-1}\mu_l^{1-2/p}2^{(3p-4)/(2p-4)}\leq 2^{(3p-4)/(2p-4)}\rho^{-1}\|f\|_{L^\infty(Q_T)}.
\end{aligned}
\]
This means, by the definition of $A(k)$,
\[
y\leq l+2^{(3p-4)/(2p-4)}\rho^{-1}\|f\|_{L^\infty(Q_T)} \qquad \text{a.e. in}\ Q_T.
\]
Take $p=3$, to conclude (\ref{2.28}).
\vskip1mm

Set $z=-y$. Then
\begin{equation}\nonumber
\left\{\begin{array}{ll}
z_t-(a(y)z_x)_x=-f, &(x,t)\in Q_T,
\\[2mm]
z(0,t)=z(1,t)=0, &t\in (0,T),
\\[2mm]
z(x,0)=-y_0(x), &x \in (0,1).
\end{array}\right.
\end{equation}
A proof analogous to that of inequality (\ref{2.28}) yields the estimate
$$
\sup_{Q_T}(-y)=\sup_{Q_T}z\leq\sup_{[0,1]}(-y_0)+\frac{6}{\rho}\|f\|_{L^\infty(Q_T)},
$$
and hence
$$
\inf_{Q_T}y=-\sup_{Q_T}(-y)\geq-\sup_{[0,1]}(-y_0)-\frac{6}{\rho}\|f\|_{L^\infty(Q_T)}=\inf_{[0,1]}y_0-\frac{6}{\rho}\|f\|_{L^\infty(Q_T)},
$$
this proves (\ref{2.29}), and so finish the proof. \hfill $\Box$

\vskip2mm

Combining (\ref{2.28}) and (\ref{2.29}), we immediately obtain

\begin{corollary}\label{Coro.2.2}
Let $y_0\in H^1_0(0,1)$, $f\in L^\infty(Q_T)$, and let $y\in L^\infty(Q_T)\cap\stackrel{\circ}{V}{\hskip-1mm}^{1,0}_2(Q_T)$ be a weak solution to the problem (\ref{2.27}). Then
\begin{equation}\label{2.34}
\sup_{Q_T}|y|\leq\sup_{[0,1]}|y_0|+\frac{6}{\rho}\|f\|_{L^\infty(Q_T)}.
\end{equation}
\end{corollary}

\section{$L^\infty$-cost estimate of control functions}
Let $A(s)=\int_0^s a(\tau)d\tau$. Then equations (\ref{1.4}) and (\ref{1.5}) are essentially the same. Now we recall the local null controllability result of
the additive control system (\ref{1.5}). In \cite{Be} and \cite{FNNV}, the authors show that for any $T>0$, there exists $\eta>0$ such that, for every $y_0\in H^1_0(0,1)$ with $\|y_0\|_{H^1(0,1)}\leq\eta$,
there exists a control $u\in L^2(Q_T)$ for which the associated state $y$ of (\ref{1.5}) satisfies $y(x,T)=0$ in $(0,1)$. Further, the control function $u$ obtained
from \cite{Be} is in $H^1([0,T]; L^2(0,1))$ and admits the estimate $\iint_{Q_T}(u^2+u_t^2)\leq C\|y_0\|^2_{L^2(0,1)}$.

However, to prove the null controllability of the multiplicative control system for the quasilinear parabolic equation (\ref{1.1}), this control cost estimate is insufficient.
We need an $L^\infty$-cost estimate for the control function when the additive control system (\ref{1.5}) achieves null controllability, which in turn allows us to obtain the
maximum modulus estimate for the state. Therefore, we provide a concise re-proof of the local null controllability for system (\ref{1.5}), with the primary aim of obtaining
an estimate for the maximum modulus of the control function. In our proof, we adopt the framework provided in \cite{Be} (which is also the technique and framework in \cite{Barbu}),
and employ an important estimate concerning the linearized system from \cite{Be}.

A fundamental step towards proving the local null controllability of the quasilinear parabolic system (\ref{1.5}) involves establishing some controllability results for its associated linearized system
\begin{equation}\label{3.1}
\left\{\begin{array}{ll}
y_{t}-(a(\tilde{y})y_x)_{x}=1_\omega u, &(x,t)\in Q_T,
\\[2mm]
y(0,t)=y(1,t)=0, &t\in(0,T),
\\[2mm]
y(x,0)=y_0(x), &x\in(0,1),
\end{array}\right.
\end{equation}
where $\tilde{y}$ is some function with $\tilde{y}_x, \sqrt{t}\tilde{y}_t\in L^\infty(Q_T)$, $\tilde{y}_{xt}\in L^2(Q_T)$, and $\tilde{y}(0,t)=\tilde{y}(1,t)=0$. For notational convenience,
from now on, $a(\tilde{y})$ will be denoted by $b$.

The following technical result, due to Fursikov and Imanuvilov \cite{FI}, is fundamental.
\begin{lemma}\label{Lemma 3.1}
There exists a function $\psi\in C^2([0,1])$ such that
$$
\psi(x)>0\ \ \mbox{in}\ (0,1), \ \ \psi(0)=\psi(1)=0, \ \ |\psi'(x)|>0\ \ \mbox{in}\ [0,1]\backslash\omega_0,
$$
where $\omega_0$ is a nonempty open set such that $\overline{\omega}_0\subset (0,1)$.
\end{lemma}

In the text below, we will use a fixed $\psi$ and a fixed $\overline{\omega}_0\subset\omega$. For any $\lambda>0$, we define
\begin{equation}\label{3.2}
\beta(x,t)=\frac{e^{\lambda\psi(x)}-e^{2\lambda\|\psi\|_{C([0,1])}}}{t(T-t)}, \ \ \phi(x,t)=\frac{e^{\lambda\psi(x)}}{t(T-t)}.
\end{equation}

Consider the adjoint system of (\ref{3.1}) in the form
\begin{equation}\label{3.3}
\left\{\begin{array}{ll}
p_{t}+(b p_x)_{x}=h, &(x,t)\in Q_T,
\\[2mm]
p(0,t)=p(1,t)=0, &t\in(0,T),
\\[2mm]
p(x,T)=p_T(x), &x\in(0,1),
\end{array}\right.
\end{equation}
where $h\in L^2(Q_T)$ and $p_T\in L^2(0,1)$.

The following Carleman inequality concerning (\ref{3.3}) was proved in \cite{Be} and \cite{FI}.
\begin{proposition}\label{Proposition 3.1}
For any solution $p$ of the adjoint system (\ref{3.3}), with $\|b_x\|_{L^\infty(Q_T)}\leq\delta$, $\|\sqrt{t}b_t\|_{L^\infty(Q_T)}$ $\leq\delta$ and for every $\lambda\geq\lambda_0(\delta)$,
$s\geq s_0(\lambda)$, it can be concluded that
\begin{align}
\begin{aligned}
&\iint_{Q_T} e^{2s\beta}\left[s^3 \phi^3 p^2+s\phi p_x^2+s^{-1}\phi^{-1}(p_t^2+p_{xx}^2)\right]dxdt
\\[1mm]
&\qquad\qquad\leq C(\lambda,\delta)\left(\int_0^T\hskip-1.5mm\int_{\omega}e^{2s\beta}s^3\phi^3 p^2dxdt+\iint_{Q_T} e^{2s\beta}h^2dxdt\right),
\end{aligned}
\label{3.4}
\end{align}
where $C(\lambda,\delta)$ is a constant independent of $p, h$, and $s$, but which may depend on $\psi$, $\lambda$, $\delta$, and $\omega$.
\end{proposition}

Proceeding as in \cite{Barbu} (cf. Corollary 1.2.1, pp. 145), one may obtain the following observability inequality.
\begin{corollary}\label{Corollary 3.1}
Under the assumptions of Proposition 3.1, we can assert that
\begin{equation}\label{3.5}
\int_0^1p^2(x,0)dx\le C(s, \lambda, \delta)\left(\int_0^T\hskip-1.5mm\int_{\omega}e^{2s\beta}\phi^3 p^2dxdt+\iint_{Q_T} e^{2s\beta}h^2dxdt\right),
\end{equation}
where $C(s, \lambda,\delta)$ is a constant that does not depend on $p$ or $h$.
\end{corollary}
\vskip2mm

Define
\begin{equation}\label{3.6}
K=\{y; \ \|y_x\|_{L^\infty(Q_T)}\leq\delta,\ \|\sqrt{t}y_t\|_{L^\infty(Q_T)}\leq\delta, \ \|y_{xt}\|_{L^2(Q_T)}\le \delta, \ y(\cdot,0)=y_0\}.
\end{equation}
Note that $K$ is compact in the $L^2(Q_T)$ topology. Indeed, it can be easily seen that $K$ is closed, and the fact $\|y\|_{H^{1}(Q_T)}$ is bounded insures its precompactness.

Consider the linearized equation (\ref{3.1}) for $\tilde{y}\in K$. By the definition of $K$, for all $\tilde{y}$, we have a fixed $\delta$ (which does not depend on $\tilde{y}$) in the Carleman inequality (\ref{3.4}).
By taking $\delta>0$ sufficiently small, we can invoke an important estimate for the solution to (\ref{3.1}), as established in \cite{Be} (see the inequality (3.34) in \cite{Be}),
\begin{equation}\label{3.7}
\begin{split}
&\|y_x\|_{L^\infty(Q_T)}^2+\|\sqrt{t}y_t\|_{L^\infty(Q_T)}^2+\|y_{xt}\|_{L^2(Q_T)}^2
\\[2mm]
&\quad\leq C\left(1+\|\sqrt{t}\tilde{y}_t\|_{L^\infty(Q_T)}^2\right)\left(1+\|\tilde{y}_x\|_{L^\infty(Q_T)}^4+\|\tilde{y}_{xt}\|_{L^2(Q_T)}^4\right)
\\[2mm]
&\qquad\cdot\left(\int_0^1\big((b y_x)_x^2(x,0)+(y_0)_x^2\big)dx+\iint_{Q_T}\big(u_t^2+u^2\big)dxdt\right)
\end{split}
\end{equation}
for any $\tilde{y}\in K$.

For brevity, in the following we will denote any constant which does not depend on $p, y, \tilde{y}, h$, and $u$ by $C$ (such a constant may, however, depend on $s, \lambda, \delta$).

For any $\varepsilon>0$, we construct the following optimal control
problem, which provides an approximate null-control for (\ref{3.1}):
\begin{equation}\label{3.8}
\mbox{Minimize}\ \left\{\iint_{Q_T}e^{-2s\beta}\phi^{-3}u^2dxdt+\frac{1}{\varepsilon}\int_0^1y^2(x,T)dx; \ u\in L^2(Q_T)\right\}\ \mbox{subject to}\ (\ref{3.1}).
\end{equation}

Applying the Pontryagin maximum principle, we find that this problem admits a unique solution $(y_\varepsilon, u_\varepsilon)$, and it follows that
\begin{equation}\label{3.9}
u_\varepsilon=1_\omega e^{2s\beta}\phi^3p_\varepsilon,
\end{equation}
where $p_\varepsilon$ is a solution of the dual system
\begin{equation}\label{3.10}
\left\{\begin{array}{ll}
(p_\varepsilon)_{t}+(b (p_\varepsilon)_x)_{x}=0, &(x,t)\in Q_T,
\\[2mm]
p_\varepsilon(0,t)=p_\varepsilon(1,t)=0, &t\in(0,T),
\\[2mm]
p_\varepsilon(x,T)=-\displaystyle\frac{1}{\varepsilon}y_\varepsilon(T,x), &x\in(0,1).
\end{array}\right.
\end{equation}

Multiplying (\ref{3.10}) by $y_\varepsilon$ and (\ref{3.1}) (where $y=y_\varepsilon$ and $u=u_\varepsilon$) by $p_\varepsilon$, adding the two equations together and integrating on $Q_T$,
it follows from (\ref{3.9}) and Corollary 3.1 that
\begin{align}
\begin{aligned}
&\int_{0}^T\hskip-1.5mm\int_\omega e^{2s\beta}\phi^3p_\varepsilon^2dxdt+\frac{1}{\varepsilon}\int_0^1y_\varepsilon^2(T,x)dx
\\[2mm]
&\qquad=-\int_0^1y_0 p_\varepsilon(x,0)dx\le \|y_0\|_{L^2(0,1)}\|p_\varepsilon(\cdot, 0)\|_{L^2(0,1)},
\end{aligned}
\nonumber
\end{align}
which implies that
\begin{equation}
\label{3.11}
\int_{0}^T\hskip-1.5mm\int_\omega e^{2s\beta}\phi^3p_\varepsilon^2dxdt+\frac{1}{\varepsilon}\int_0^1y_\varepsilon^2(T,x)dx
\le C\int_0^1y_0^2dx.
\end{equation}
By (\ref{3.9}), (\ref{3.11}) and Proposition 3.1, we derive
\begin{equation}\label{3.12}
\iint_{Q_T}\big(u_\varepsilon^2+(u_\varepsilon)_t^2\big)dxdt\le C_1\int_0^1y_0^2dx,
\end{equation}
where $C_1$ is a constant depending on $s,\lambda, \delta$ and $\omega$.

Further, we estimate the maximum modulus of $u_\varepsilon$. The method borrows from the iteration method used in \cite{Barbu} in order to get an $L^\infty$-control
function for the semilinear heat equations, which is also employed in \cite{LZ} to obtain an estimate for the control function in certain H\"{o}lder space. Write
\[
\phi_0(t)=\frac{1}{t(T-t)}, \ \beta_0(t)=\big(1-e^{2\lambda\|\psi\|_{C([0,1])}}\big)\phi_0(t),\ v_\varepsilon=e^{(s+\delta_0)\beta_0(t)}\phi_0^3(t)p_\varepsilon,
\]
where $0<\delta_0<s/2$. Then it is easy to check that $v_\varepsilon$ satisfies
\begin{equation}\label{3.13}
\left\{\begin{array}{ll}
(v_\varepsilon)_t+(b(v_\varepsilon)_x)_x=h_\varepsilon, &(x,t)\in Q_T,
\\[2mm]
v_\varepsilon(0,t)=v_\varepsilon(1,t)=0, &t\in (0,T),
\\[2mm]
v_\varepsilon(x,0)=v_\varepsilon(x,T)=0, &x\in (0,1),
\end{array}\right.
\end{equation}
where $h_\varepsilon=p_\varepsilon\big(e^{(s+\delta_0)\beta_0}\phi_0^3\big)_t$.

A simple calculation shows that
\begin{equation}\label{3.14}
\left\{\begin{array}{ll}
\big(e^{(s+\delta_0)\beta_0}\phi_0^3\big)_t=e^{(s+\delta_0)\beta_0}\big[(s+\delta_0)\phi_0^3 \beta_0^{\prime}+3\phi_0^2\phi_0^{\prime}\big],
\\[3mm]
\big|\big(e^{(s+\delta_0)\beta_0}\phi_0^3\big)_t\big|^2\le e^{2(s+\delta_0)\beta_0}\big[2(s+\delta_0)^2 \phi_0^6 (\beta_0^{\prime})^2+18\phi_0^4(\phi_0^{\prime})^2\big],
\\[3mm]
\phi_0^{\prime}=\displaystyle\frac{2t-T}{t^2(T-t)^2}, \ ({\phi_0}')^2\le T^2\phi_0^4,
\\[3mm]
\beta_0\leq\beta\leq\displaystyle\frac{e^{\lambda\|\psi\|_{C([0,1])}}}{e^{\lambda\|\psi\|_{C([0,1])}}+1} \beta_0<0,\ ({\beta_0}')^2\le CT^2\phi_0^4,
\\[3mm]
0<\phi_0 \leq\phi\leq e^{\lambda\|\psi\|_{C([0,1])}}\phi_0,
\\[3mm]
e^{s\beta}\phi_0^{k}\le C<\infty \ \mbox{for all}\ k\in \mathbb R.
\end{array}\right.
\end{equation}
Thus
\begin{equation}\label{3.15}
\begin{split}
\|h_\varepsilon\|_{L^2(Q_T)}^2&\leq C\iint_{Q_T}e^{2(s+\delta_0)\beta_0}s^2\phi_0^{10}p_\varepsilon^2dxdt
\\[1mm]
&\leq C\iint_{Q_T}e^{2s\beta}s^3\phi^3(e^{2\delta_0\beta_0}\phi_0^{7})p_\varepsilon^2dxdt
\\[1mm]
&\leq C\iint_{Q_T}e^{2s\beta}s^3\phi^3p_\varepsilon^2dxdt.
\end{split}
\end{equation}
By Proposition 3.1, combining (\ref{3.11}) and (\ref{3.15}), we get
\begin{equation}\label{3.16}
\|h_\varepsilon\|_{L^2(Q_T)}^2\leq C\int_0^T\hskip-1mm\int_\omega e^{2s\beta}s^3\phi^3 p_\varepsilon^2dxdt\le C\|y_0\|_{L^2(0,1)}^2.
\end{equation}
According to the $L^p$-theory of the linear parabolic equations
(see Theorem 7.17 in \cite{Lie} page 176), we find that the solution of (\ref{3.13}) satisfies
\begin{equation}\nonumber
\|v_\varepsilon\|_{W^{2,1}_2(Q_T)}\le C\|h_\varepsilon\|_{L^2(Q_T)}\le C\|y_0\|_{L^2(0,1)}.
\end{equation}
Using the continuous embedding in Sobolev space $W^{2,1}_2(Q_T)$, we see that
\begin{equation}\label{3.17}
\|v_\varepsilon\|_{L^\infty(Q_T)}\le \|v_\varepsilon\|_{C^{1/2,1/4}(\overline{Q}_T)}\le C\|v_\varepsilon\|_{W^{2,1}_2(Q_T)}\le C\|y_0\|_{L^2(0,1)}.
\end{equation}
Note that
$$
u_\varepsilon=1_\omega e^{2s\beta}\phi^3p_\varepsilon=1_\omega e^{2s\beta}\phi^3e^{-(s+\delta_0)\beta_0}\phi_0^{-3}v_\varepsilon,
$$
and $e^{2s\beta}e^{-(s+\delta_0)\beta_0}\phi^3\phi_0^{-3}\le C$ (here we take $0<\delta_0<s/2$, and $\lambda\|\psi\|_{C([0,1])}\ge 2$). By (\ref{3.17}), we conclude that
\begin{equation}\label{3.18}
\|u_\varepsilon\|_{L^\infty(Q_T)}\le C_2\|y_0\|_{L^2(0,1)},
\end{equation}
where $C_2$ is a constant depending on $s, \lambda, \delta$ and $\omega$.

Combining (\ref{3.7}) and (\ref{3.12}) and recalling $\tilde{y}\in K$ now yields
\begin{equation}\label{3.19}
\begin{split}
&\|(y_\varepsilon)_x\|_{L^\infty(Q_T)}^2+\|\sqrt{t}(y_\varepsilon)_t\|_{L^\infty(Q_T)}^2+\|(y_\varepsilon)_{xt}\|_{L^2(Q_T)}^2
\\[2mm]
&\quad\leq C(1+\delta^6)\left(\int_0^1\big((b y_x)_x^2(x,0)+(y_0)_x^2+y_0^2\big)dx\right)
\\[2mm]
&\quad=C(1+\delta^6)\left(\int_0^1\big(A(y_0)\big)_{xx}^2+(y_0)_x^2+y_0^2\big)dx\right),
\end{split}
\end{equation}
where $A(y_0)=\int_0^{y_0}a(\tau)d\tau$. By making
\begin{equation}\label{3.20}
\int_0^1\big[(A(y_0))_{xx}^2+(y_0)_x^2+y_0^2\big]dx<\eta_0^2
\end{equation}
for a sufficiently small $\eta_0>0$, we obtain that $y_\varepsilon\in K$.

Choose a sequence of $(y_\varepsilon, u_\varepsilon)$, $\varepsilon\rightarrow 0^+$, that achieves the optimum in (\ref{3.8}). From the estimates (\ref{3.12}), (\ref{3.18}), it follows that (on a subsequence)
\begin{align}
\begin{aligned}
&u_\varepsilon \to u \quad &&\text{weakly in}\ L^2(Q_T),
\\[1mm]
&u_\varepsilon \to u \quad &&\text{weakly star in}\ L^{\infty}(Q_T),
\\[1mm]
&(u_\varepsilon)_t \to v \quad &&\text{weakly in}\ L^2(Q_T),
\\[1mm]
&y_\varepsilon \to y \quad &&\text{strongly in}\ L^2(Q_T).
\end{aligned}
\nonumber
\end{align}
Obviously, $v=u_t$. By (\ref{3.19}), other immediate consequences are that (on a subsequence)
\begin{align}
\begin{aligned}
&(y_\varepsilon)_t \to y_t \quad &&\text{strongly in}\ H^{-1}(Q_T),
\\[1mm]
&(b(y_\varepsilon)_x)_x \to (by_x)_x \quad &&\text{strongly in}\ L^2((0,T); H^{-2}(0,1)),
\end{aligned}
\nonumber
\end{align}
as $\varepsilon\to 0^+$. Thus $(y,u)$ satisfies the linearized equation (\ref{3.1}). Applying the weak lower semicontinuity of the norm to (\ref{3.19}),
we obtain the following estimates provided that (\ref{3.20}) holds for some sufficiently small $\eta_0$:
$$
\|y_x\|_{L^\infty(Q_T)}\leq\delta,\ \|\sqrt{t}y_t\|_{L^\infty(Q_T)}\leq\delta, \ \|y_{xt}\|_{L^2(Q_T)}\le \delta,
$$
i.e., $y\in K$. Now, letting $\varepsilon\to 0$ in (\ref{3.11}), we see that $y(\cdot, T)=0$ in $(0,1)$. Finally, (\ref{3.12}) directly yields
$\iint_{Q_T}(u^2+u_t^2)dxdt\leq C_1\|y_0\|_{L^2}^2$.

\vskip2mm

We are now ready to apply Kakutani's theorem. Consider $\Phi: K\to 2^K$,
\begin{equation}\label{3.21}
\begin{split}
\Phi(\tilde{y})=\biggl\{&y; \ y\in K,\ y(\cdot, T)=0,\ \text{and}\ \exists u\in L^\infty(Q_T)\cap H^1([0,T]; L^2(0,1)), \text{with}\ \|u\|_{L^\infty(Q_T)}
\\
&\le C_2\|y_0\|_{L^2(0,1)} \ \text{and}\ \iint_{Q_T}(u^2+u_t^2)dxdt\leq C_1\|y_0\|_{L^2(0,1)}^2,\ \text{such that}\ (y,u)\ \text{satisfies}\ (\ref{3.1})
\biggr\}.
\end{split}
\end{equation}

Clearly, $\Phi$ is well defined,takes nonempty values for every $\tilde{y}\in K$, and has convex values. In order to prove that $\Phi(\tilde{y})$ is closed, consider
a sequence $\{y_n\}_{n=1}^\infty\subset \Phi(\tilde{y})$ such that
\begin{equation}\label{3.22}
y_n\rightarrow y \qquad \text{strongly in}\ L^2(Q_T)
\end{equation}
as $n\to \infty$, and choose the corresponding $u_n$ as in the definition of $\Phi$ in (\ref{3.21}). Namely, $(y_n, u_n)$ satisfy
\begin{equation}\label{3.23}
\left\{\begin{array}{ll}
(y_n)_{t}-(a(\tilde{y})(y_n)_x)_{x}=1_\omega u_n, &(x,t)\in Q_T,
\\[2mm]
y_n(0,t)=y_n(1,t)=0, &t\in(0,T),
\\[2mm]
y_n(x,0)=y_0(x), &x\in(0,1).
\end{array}\right.
\end{equation}

Since $\iint_{Q_T}(u_n^2+(u_n)_t^2)dxdt\leq C_1\|y_0\|_{L^2}^2$ and $\|u_n\|_{L^\infty(Q_T)}\le C_2\|y_0\|_{L^2(0,1)}$, we may select a convergent subsequence $\{u_{n_j}\}_{j=1}^\infty$
and a limit function $u^*$ such that
\begin{align}
\begin{aligned}
&u_{n_j} \to u^* \quad &&\text{weakly in}\ L^2(Q_T),
\\[1mm]
&u_{n_j} \to u^* \quad &&\text{weakly star in}\ L^{\infty}(Q_T),
\\[1mm]
&(u_{n_j})_t \to u^*_t \quad &&\text{weakly in}\ L^2(Q_T)
\end{aligned}
\nonumber
\end{align}
as $j\to\infty$. Clearly, $\iint_{Q_T}({u^*}^2+{u^*_t}^2)dxdt\leq C_1\|y_0\|^2_{L^2(0,1)}$ and $\|u^*\|_{L^\infty(Q_T)}\le C_2\|y_0\|_{L^2(0,1)}$.

\vskip1mm

Taking $u_n=u_{n_j}$ in (\ref{3.23}) and applying estimate (\ref{3.19}) to the corresponding solutions $y_{n_j}$, we know that there exists
a subsequence of $\{y_{n_j}\}_{j=1}^\infty$ (for notational simplicity, we still denote this subsequence by $\{y_{n_j}\}_{j=1}^\infty$), such that $(y_{n_j})_t\rightarrow z$
strongly in $H^{-1}(Q_T)$ and $(b(y_{n_j})_x)_x\rightarrow w$ strongly in $L^2((0,T); H^{-2}(0,1))$. For any $\varphi\in C_0^\infty(\overline{Q}_T)$, we have
$$
\iint_{Q_T}(b(y_{n_j})_x)_x\varphi dxdt=-\iint_{Q_T}b(y_{n_j})_x\varphi_x dxdt=\iint_{Q_T}y_{n_j}(b\varphi_x)_x dxdt.
$$
Letting $j\to +\infty$, we find
$$
\iint_{Q_T}w\varphi dxdt=\iint_{Q_T}y(b\varphi_x)_x dxdt=-\iint_{Q_T}y_x b\varphi_xdxdt=\iint_{Q_T}(b y_x)_x \varphi dxdt
$$
for all $\varphi\in C_0^\infty(\overline{Q}_T)$, whence $(by_x)_x=w$. A similar argument shows that $y_t=z$. Applying estimate (\ref{3.19}) again shows $\|y_x\|_{L^\infty(Q_T)}\leq\delta$,
$\|\sqrt{t}y_t\|_{L^\infty(Q_T)}\leq\delta$ and $\|y_{xt}\|_{L^2(Q_T)}\le \delta$ provided that (\ref{3.20}) holds for some sufficiently small $\eta_0$, i.e. $y\in K$. By passing to the weak limit in (\ref{3.23}),
satisfied by $(y_{n_j}, u_{n_j})$, we obtain that $(y,u^*)$ also satisfies the equation for the same $\tilde{y}$. Therefore, $y\in \Phi(\tilde{y})$, and thus $\Phi(\tilde{y})$ is closed.
Since $\Phi(\tilde{y})\subset K$, it follows that $\Phi(\tilde{y})$ is compact for every $\tilde{y}\in K$. Since $y_n(\cdot, 0)=y(\cdot, 0)=y_0(\cdot)$ and $y_n(\cdot,T)=0$, and furthermore,
\begin{equation}\label{3.24}
\begin{split}
\|y_n(\cdot, T)-y(\cdot, T)\|_{L^2(0,1)}^2&=\int_0^1\hskip-1.5mm\int_0^T\frac{d}{dt}(y_n(x,t)-y(x,t))^2dxdt
\\[2mm]
&\le2\iint_{Q_T}|y_n(x,t)-y(x,t)(y_n(x,t)-y(x,t))_t|dxdt
\\[2mm]
&\le2\|y_n-y\|^{1/2}_{H^1([0,T]; L^2(0,1))}\|y_n-y\|^{1/2}_{L^2(Q_T)}
\\[2mm]
&\le 2\delta^{1/2}\|y_n-y\|^{1/2}_{L^2(Q_T)}.
\end{split}
\end{equation}
Passing $n\to \infty$ and invoking (\ref{3.22}), yields $y(\cdot, T)=0$ a.e. in $(0,1)$.

In such a case, the lower semicontinuity of $\Phi$ can be obtained from the fact it has a closed graph. Indeed, if $\tilde{y}_n\in K$, $\tilde{y}_n\rightarrow \tilde{y}$, and
$y_n\in \Phi(\tilde{y}_n)\rightarrow y$ strongly in $L^2(Q_T)$ as $n\to \infty$, consider the corresponding control function $u_n$, as in (\ref{3.23}) with $\tilde{y}$ being replaced by $\tilde{y}_n$,
and we obtain (on a subsequence) that
\allowdisplaybreaks
\begin{align}
\begin{aligned}
&u_n \to u^* \quad &&\text{weakly in}\ L^2(Q_T),
\\[1mm]
&u_n \to u^* \quad &&\text{weakly star in}\ L^{\infty}(Q_T),
\\[1mm]
&(u_n)_t \to u^*_t \quad &&\text{weakly in}\ L^2(Q_T),
\\[1mm]
&(y_n)_t \to y_t \quad &&\text{strongly in}\ H^{-1}(Q_T),
\\[1mm]
&(y_n)_x \to y_x \quad &&\text{strongly in}\ L^2((0,T); H^{-1}(0,1)),
\\[1mm]
&\tilde{y}_n \to y \quad &&\text{strongly in}\ L^2(Q_T),
\end{aligned}
\label{3.25}
\end{align}
as $n\to\infty$. By (\ref{1.2}), $|a(\tilde{y}_n)-a(\tilde{y})|\leq M|\tilde{y}_n-\tilde{y}|$, it follows that
$$
a(\tilde{y}_n)\to a(\tilde{y}) \ \ \quad \text{as} \ n \ \to \infty
$$
strongly in $L^2(Q_T)$, and weakly-star in $L^\infty(Q_T)$. Then
$$
a(\tilde{y}_n)(y_n)_x\to a(\tilde{y})y_x\quad \text{weakly in} \ L^2((0,T); H^{-1}(0,1)),
$$
or, equivalently,
$$
(a'(\tilde{y}_n)(y_n)_x)_x\to(a'(\tilde{y})y_x)_x\quad \text{weakly in}\ L^2((0,T); H^{-2}(0,1))
$$
as $n\to\infty$. By going to the weak limit in (\ref{3.23}), we obtain that the pair $(y, u^*)$ also satisfies the linearized system, with $b=a(\tilde{y})$. The other conditions ($\tilde{y}, y\in K$, $y(\cdot, T)=0$)
are obviously satisfied (see above the details of the proof), so $(\tilde{y}, y)$ belongs to the graph of $\Phi$.

Now we can apply Kakutani's theorem and obtain that there is $y\in K$ such that $y\in\Phi(y)$. Such a $y$ is a solution of the diffusion equation (\ref{1.5}) with $y(\cdot, 0)=y_0$
and $y(\cdot,T)=0$. In addition, its controller $u$ satisfies the required estimate.

Recall (\ref{3.20}). we have now proved the null controllability of system (\ref{1.5}) under the condition that $\int_0^1\big[(A(y_0))_{xx}^2+(y_0)_x^2+y_0^2\big]dx<\eta_0^2$ for a sufficiently small $\eta_0$. By Proposition 2.3,
$$
\int_0^1\big[(A(y(x,t_0)))_{xx}^2+y^2_x(x,t_0)+y^2(x,t_0)\big]dx\le C\int_0^1 (y_0^2+(y_0)_x^2)dx.
$$
Therefore, the following estimate holds:
$$
\int_0^1\big[(A(y(x,t_0)))_{xx}^2+y^2_x(x,t_0)+y^2(x,t_0)\big]dx<\eta_0^2,
$$
provided that the initial state $y_0\in H^1_0(0,1)$ with $\|y_0\|_{H^1(0,1)}\leq\eta$ for some (sufficiently small) $\eta>0$.

Define
\begin{equation}\label{3.26}
u(x,t)=
\left\{\begin{array}{ll}
0, & \mbox{in}\ \ (0,1)\times(0,t_0),
\\[2mm]
u^*, & \mbox{in}\ \ (0,1)\times(t_0,T),
\end{array}\right.
\end{equation}
where $u^*$ is a controller for $y$ on $(t_0, T)$.

Thus, we have proved the main result of the present section.
\begin{theorem}\label{Theorem 3.1}
Suppose that the diffusion coefficient $a(s)$ satisfies conditions (\ref{1.2}) and (\ref{1.3}). For any $T>0$, there exists $\eta>0$ such that,
for every $y_0\in H^1_0(0,1)$ with $\|y_0\|_{H^1(0,1)}\leq\eta$, there exist a control $u\in H^1([0,T]; L^2(0,1))\cap L^\infty(Q_T)$, such that the associated state $y$ of (\ref{1.5}) satisfies $y(x,T)=0$ in $(0,1)$.
Moreover,
\begin{equation}
\iint_{Q_T}(u^2+u_t^2)dxdt\leq C\|y_0\|^2_{L^2(0,1)},
\end{equation}
and
\begin{equation}
\|u\|_{L^\infty(Q_T)}\le C\|y_0\|_{L^2(0,1)}.
\end{equation}
\end{theorem}
where $C$ is a constant depending on $\rho, \kappa, M$ (the constants involved in conditions (\ref{1.2}) and (\ref{1.3})), and $\omega$.

\section{Proofs of Theorems 1.1-1.2}
Based on the decay estimates and the maximum modulus estimates of solutions to quasilinear parabolic equations, together with the local null controllability of
quasilinear parabolic equations under additive controls, we establish the proofs of Theorems 1.1 and 1.2 in this section.
\vskip1mm

{\bf Proof of Theorem 1.1.}\ Take $t_1=32M^2C_0^4/\rho^3$ ($M, \rho$ are constants appearing in (\ref{1.2}) and (\ref{1.3}), $C_0$ is a constant appearing in (\ref{2.13})). In light of Proposition 2.1 then,
we see that the solution of (\ref{2.1}) satisfies
\begin{equation}\label{4.1}
\|y(\cdot,t)\|_{L^\infty(0,1)}\le \frac{\rho}{4MC_0^2} \qquad \forall t\ge t_1.
\end{equation}

Let $y(\cdot, t_1)$ be the new initial datum and consider the following system
\begin{equation}\label{4.2}
\left\{\begin{array}{ll}
\bar{y}_{t}-(a(\bar{y}){\bar y}_x)_{x}=0, &\quad (x,t)\in Q_T,
\\[2mm]
\bar{y}(0,t)=\bar{y}(1,t)=0, &\quad t\in(0,T),
\\[2mm]
\bar{y}(x,0)=y(x,t_1), &\quad x\in(0,1).
\end{array}\right.
\end{equation}
By virtue of (\ref{4.1}) and Proposition 2.2, there exists a sufficiently large $t_2>0$ such that
\begin{equation}\label{4.3}
\|\bar{y}(\cdot, t_2)\|_{H^1(0,1)}<\eta.
\end{equation}
Therefore, if we set $u\equiv 0$ in (\ref{1.1}), we obtain a corresponding solution $y$. This solution $y$ is also a solution to (\ref{1.5}) under the condition $u\equiv 0$,
and it simultaneously satisfies (\ref{2.1}). Consequently, it holds that
\begin{equation}\label{4.4}
\|y(\cdot, t_1+t_2)\|_{H^1(0,1)}<\eta.
\end{equation}

For any $t_3>0$, let us deal with the following additive control system
\begin{equation}\label{4.5}
\left\{\begin{array}{ll}
Y_{t}-(a(Y)Y_x)_{x}=1_\omega v, &\quad \text{in} \ \ Q_{t_3},
\\[2mm]
Y(0,t)=Y(1,t)=0, &\quad \text{in} \ \ (0,t_3),
\\[2mm]
Y(\cdot,0)=y(\cdot, t_1+t_2), &\quad \text{in} \ \ (0,1).
\end{array}\right.
\end{equation}
Here $y$ is the solution of (\ref{1.1}) with $u\equiv 0$, in other words, $y$ is the solution of (\ref{2.1}).

Applying Theorem 3.1, there exists a control function $v\in L^{\infty}(Q_{t_3})$ such that the corresponding solution to (\ref{4.5}) satisfies
$Y(x, t_3)=0$ in $(0,1)$. Moreover, $v$ can be chosen such that the following estimate holds:
\begin{align}
\label{4.6} \|v\|_{L^\infty(Q_{t_3})}\le C\|y(\cdot, t_1+t_2)\|_{L^2(0,1)},
\end{align}
where $C$ is a constant independent of $t_1, t_2$.

Based on the maximum modulus estimate valid for weak solutions to quasilinear parabolic equations (\ref{2.27}) (see Corollary 2.2), it follows from Proposition 2.1 and (\ref{4.6}) that
\begin{equation}\label{4.7}
\begin{split}
\sup_{Q_T}|Y|&\leq\sup_{[0,1]}|y(\cdot, t_1+t_2)|+\frac{6}{\rho}\|v\|_{L^\infty(Q_{t_3})}
\\[2mm]
&\leq\sup_{[0,1]}|y(\cdot, t_1+t_2)|+\frac{6}{\rho}\sup_{[0,1]}|y(\cdot, t_1+t_2)|
\\[2mm]
&\leq C\|y_0\|_{L^2(0,1)}(t_1+t_2)^{-1/2}.
\end{split}
\end{equation}
Note that $g\in C(\mathbb R)$ and $g(0)$ $=0$. By (\ref{4.7}),
we can take $t_2$ large enough such that
\begin{equation}
\label{4.8} \|g(Y)\|_{L^{\infty}(Q_{t_3})}\le \theta_0/2.
\end{equation}
Denote $\tilde{u}=\frac{v}{g(Y)-\theta}$ in $Q_{t_3}$. Note that $y(\cdot, t_1+t_2)\in H_0^1(0,1)$ and $\|y(\cdot, t_1+t_2)\|_{H^1(0,1)}<\eta$
can be made arbitrarily small by choosing $t_2$ sufficiently large. Therefore, we see that the following problem
\begin{align}
\label{4.9} \left\{\begin{array}{ll} z_{t}-(a(z)z_x)_x=1_\omega u(g(z)-\theta)
&\mbox{in}\ \ Q_{t_3},
\\[2mm]
z(0,t)=z(1,t)=0 &\mbox{in}\ \ (0,t_3),
\\[2mm]
z(\cdot, 0)=y(\cdot, t_1+t_2) &\mbox{in} \ \ (0,1)
\end{array}\right.
\end{align}
admits uniquely a solution $z\in W^{2,1}_2(Q_{t_3})$, and $\hat u=\tilde{u}$ can act as an
input control for this system. Clearly, the solution $Y$ of (\ref{4.5}) is also a solution of
(\ref{4.9}) with $\hat u=\tilde{u}$. Using the uniqueness of
solutions to (\ref{4.5}), one can see easily that
the solution $z$ to (\ref{4.9}) with the control $\hat u$ and the solution $Y$ to (\ref{4.5})
with the control $v$ become identical in $Q_{t_3}$.
Recall that $Y(\cdot,t_3)=0$ in $(0,1)$. We have $z(\cdot, t_3)=0$ in $(0,1)$.

Given $y_0\in H^1_0(0,1)$. From the above argument, we can
select $T>0$ large enough such that the corresponding solution $y$
to (\ref{1.1}) with the control
\begin{align}
\label{4.10} u=\left\{\begin{array}{ll} 0,
&\ \ \mbox{in}\ \ (0,1)\times(0, t_1+t_2)
\\[2mm]
\tilde{u}, &\ \ \mbox{in}\ \ (0,1)\times(t_1+t_2, T)
\end{array}\right.
\end{align}
satisfies $y(\cdot, T)=0$ in $(0,1)$. The proof is complete.\quad$\Box$
\vskip2mm

\begin{remark}
By examining the proof of Theorem 1.1, it is not hard to see that if $t_2$ is taken sufficiently large, then $\|u\|_{L^\infty(Q_T)}$ can be made arbitrarily small.
Therefore, for any $\sigma>0$, there exist a time $T>0$ and a control function $u$ with $\|u\|_{L^\infty(Q_T)}<\sigma$, such that the corresponding state of (\ref{1.1})
satisfies $y(\cdot, T)=0$ a.e. in $(0,1)$.
\end{remark}

{\bf Proof of Theorem 1.2.}\ From the proof of Theorem 1.1, it is immediately clear that if we take the control function $u$ in the system (\ref{1.5})
as given in (\ref{4.10}), then the solution $y$ of (\ref{1.5}) satisfies $y(x,T)=0$ in $(0,1)$. The proof is complete. \quad$\Box$

\section{Time Optimal Control}
Based on the null controllability of the system (\ref{1.1}), we now prove the existence of the time optimal control.

Since no growth restriction is imposed on the nonlinearity $g(s)$, in general, (\ref{1.1}) has no globally defined
(in time) solution, and the solution may blow up in finite
time. Therefore, we first study the existence time of the local solution.

\begin{lemma}
Let $\theta\in L^\infty(Q_\infty)$ and $y_0\in
H^1_0(0,1)$. If $g$ satisfies the assumptions in Theorem 1.1,
then there exists a constant $\varepsilon_0>0$ depending on $g$,
$\theta$ and $y_0$, such that if
$\|u\|_{L^\infty(Q_\infty)}\le \varepsilon_0$, then the problem (\ref{1.1}) has a unique solution $y$ in the space
$\stackrel{\circ}{W}{\hskip-1mm}^{1,1}_2(Q_{T_0})\cap L^\infty(Q_{T_0})$, where $T_0$ is a constant and $T_0\ge 1$.
\end{lemma}
{\bf Proof.}\quad Let $j_n$ be a standard mollifying sequence in
$\mathbb R$, namely, $j_n(s)=\frac{1}{n}j(\frac{s}{n})$ with
$j(s)\ge 0$, $j(s)\in C_0^\infty({\mathbb{R}})$,
supp$j(s)\subset[-1, 1]$ and $\int_{-\infty}^{+\infty} j(s)ds=1$.

For fixed $u\in L^\infty(Q_\infty)$ and $\theta\in
L^\infty(Q_\infty)$, consider the following problem
\begin{align}\label{5.1}
\left\{\begin{array}{ll} y_{t}-(a(y)y_x)_x=u_n g_n(y)-u_n\theta_n, &\mbox{in}\ \ Q_T,
\\[2mm]
y(0,t)=y(1,t)=0, &\mbox{in}\ \ (0,T),
\\[2mm]
y{(x,0)}=y_{0n}(x), &\mbox{in}\ \ (0,1).
\end{array}\right.
\end{align}
We define $g_n$, $u_n$, $\theta_n$ and $y_{0n}$ as follows. Put
\begin{align}
g_n(s)&=j_n\ast\max(-n, \min(g(\tau), n))
\nonumber
\\
&\triangleq\int_{\mathbb R}j_n(s-\tau)\max(-n, \min(g(\tau),
n))d\tau. \nonumber
\end{align}
It is easily verified that $g_n\in C^\infty(\mathbb{R})$ and
$|g_n|\le n$. Similarly, we can define $u_n$, $\theta_n$ and $y_{0n}$
by using the above convolution operator $\ast$. For example, define
\begin{align}
u_n(x,t)=
\iint_{\mathbb R^{2}}&\hskip-1mmj_n(x-y)j_n(t-\tau) \max(-n, \min(1_\omega u(y,\tau),n))dyd\tau.
\nonumber
\end{align}
Then $u_n\in C^\infty(Q_\infty)$ with $u_n\equiv0$ in $\omega_n\times\mathbb{R}^+$,
where $\omega_n=\{x\in\Omega\backslash\omega;\ \mbox{dist}(x, \omega)
<\frac{1}{n}\}$, and for any $p\in [1, +\infty)$,
\begin{equation}
\label{5.2}
\|u_n-1_\omega u\|_{L^p(Q_\infty)}\rightarrow 0 \quad\ \ \mbox{as}\ \ \ n\rightarrow +\infty.
\end{equation}
Likewise, $\theta_n\in C^\infty(Q_\infty)$, $y_{0n}\in C^\infty_0[0,1]$, and satisfy
\begin{equation}
\label{5.3}
\|\theta_n-\theta\|_{L^\infty(Q_\infty)}\rightarrow 0, \
\ \|y_{0n}-y_0\|_{H^1(0,1)}\rightarrow 0
\end{equation}
as $n\rightarrow +\infty$.

With these choices of $g_n$, $u_n$, $\theta_n$ and $y_{0n}$, we
see that the approximate problem (\ref{5.1}) admits a local
classical solution on $Q_{T_n}=(0,1)\times(0, T_n)$ for all $n$ (see
\cite{LSU}, Chap. V), which we denote by $y_n$. Here $T_n$ is the
maximal existence time of the corresponding local solution.

Now we claim that there exists ${T_0}>1$ such that
\begin{equation}
\|y_n\|_{L^\infty(Q_{T_0})}\le A \ \ \mbox{for all}\ n,
\label{5.4}
\end{equation}
where $A$ is a constant depending upon
$\|u\|_{L^\infty(Q_\infty)}$,
$\|\theta\|_{L^\infty(Q_\infty)}$,
$\|y_0\|_{L^\infty(0,1)}$, and $g$.

Recall that $g(\cdot)\in C(\mathbb{R})$, and $g(\cdot)$ is Lipschitz continuous on $[-L, L]$ for any $L>0$. Using the above approximate method, we can take $G\in C^1(\mathbb R)$
such that $|g(s)|\le G(s)$. Let $g^+$ and $g^-$ be the solution of the ordinary differential equations
\begin{equation}
g'=K_1G(g)+K_2, \quad g(0)= \|y_0\|_{L^\infty(0,1)} \nonumber
\end{equation}
and
\begin{equation}
g'=-K_1G(g)-K_2, \quad g(0)=-\|y_0\|_{L^\infty(0,1)} \nonumber
\end{equation}
respectively, where $K_1\ge 0$ and $K_2\ge 0$ are constants. Clearly, $g^+$ is increasing in $t$, and $g^-$ is decreasing in $t$.
Furthermore, the smaller the values of $K_1\ge 0$ and $K_2\ge 0$, the smaller the value of $g^{+}(t)$ and the smaller the value of $-g^{-}(t)$.

By the standard theory of ordinary differential equations (see
\cite{CL}, Theorem 2.2, pp. 10) there exists ${T_0}>0$ such that $g^+(t)$ and
$g^-(t)$ exist on $[0, {T_0}]$, where ${T_0}$ depends on
$\|y_0\|_{L^\infty(0,1)}$, $K_1$, $K_2$ and the function
$G(\cdot)$. Moreover, we can select $K_1$ and $K_2$ small enough
such that ${T_0}\ge 1$.

For the above selected $K_1$ and $K_2$, we take
$\|u\|_{L^\infty(Q_\infty)}<\varepsilon_0$ such that
$\|u_n\|_{L^\infty(Q_\infty)}\le K_1$ and $
\|u_n\theta_n\|_{L^\infty(Q_\infty)}\le K_2$. By standard
comparison theorems,
\begin{equation}
\nonumber|y_n(x,t)|\le \max\{g^+(t), -g^-(t)\}\quad \mbox{in}\
(0,1)\times(0, {T_0}).
\end{equation}
Setting $A=\max\{g^+({T_0}), -g^-({T_0})\}$, we see that (\ref{5.4}) holds.

Multiplying (\ref{5.1}) by $y_n$ (where $Q_T$ is replaced by
$Q_{{T_0}}$) and integrating the resulting relation over $Q_{t}$, $0<t\le {T_0}$,
we deduce from Cauchy's inequality with $\varepsilon$ that
\begin{equation}
\begin{split}
&\frac{1}{2}\int_0^1y_n^2(x,t)dx+\int_0^t\hskip-1mm\int_0^1 a(y_n)(y_n)_x^2dxds
\\[2mm]
=&\int_0^t\hskip-1mm\int_0^1 u_n g_n(y_n)y_ndxds-\int_0^t\hskip-1mm\int_0^1 u_n \theta_n y_n dxds+\frac{1}{2}\int_0^1y_{0n}^2dx
\\[2mm]
\nonumber
\le&\left(\int_0^t\hskip-1mm\int_0^1 |u_n g_n(y_n)|^2dxds\right)^{1/2}\left(\int_0^t\hskip-1mm\int_0^1 y_n^2dxds\right)^{1/2}
\\[2mm]
&\ \ +\left(\int_0^t\hskip-1mm\int_0^1 u_n^2 \theta_n^2dxds\right)^{1/2}\left(\int_0^t\hskip-1mm\int_0^1y_n^2dxds\right)^{1/2}+\frac{1}{2}\int_0^1y_{0n}^2dx
\\[2mm]
\le &\varepsilon\int_0^t\hskip-1mm\int_0^1y_n^2dxds+C(\varepsilon)\Big(K_1\sup_{[-A, A]}|g_n|+K_2\Big)+\frac{1}{2}\int_0^1y_{0n}^2dx
\end{split}
\end{equation}
for any $\varepsilon>0$. Take $\varepsilon=\pi^2\rho/2$. By (\ref{1.3}) and Poincar$\acute{\mbox{e}}$'s inequality, we obtain
\begin{equation}\nonumber
\int_0^1y_n^2(x,t)dx+\rho\int_0^t\hskip-1mm\int_0^1 (y_n)_x^2dxds
\le C(\rho)\Big(K_1\sup_{[-A, A]}|g_n|+K_2\Big)+\int_0^1y_{0n}^2dx.
\end{equation}
The construction of $g_n$ and $y_{0n}$ directly implies that
\begin{equation}\label{5.5}
\sup_{0\le t\le{T_0}}\int_0^1y_n^2(x,t)dx+\iint_{Q_{T_0}}(y_n)_x^2dxds \le A_0,
\end{equation}
where $A_0$ is a constant independent of $n$. Since
$$
\iint_{Q_{T_0}}(A(y_n))_x^2dxdt=\iint_{Q_{T_0}}a^2(y_n)(y_n)_x^2dxdt\le \kappa^2\iint_{Q_{T_0}}(y_n)_x^2dxdt,
$$
where $\kappa$ follows from the condition (\ref{1.3}), we see that
\begin{equation}\label{5.6}
\iint_{Q_{T_0}}(A(y_n))_x^2dxds \le \kappa^2A_0.
\end{equation}

By virtue of (\ref{5.4}), (\ref{5.5}) and (\ref{5.6}), there exist a subsequence
$n_j\rightarrow+\infty$ as $j\to +\infty$, a function $y\in C(0, T_0; L^2(0,1))\cap L^2(0, T_0; H^1_0(0,1))$ and a function $v\in L^2(Q_{T_0})$, such
that
\begin{align}
\begin{aligned}
&y_{n_j}\rightarrow y &&\quad\mbox{weakly star in}\
L^\infty(Q_{T_0}),
\\[1mm]
&(y_{n_j})_x\rightarrow y_x&&\quad\mbox{weakly in}\
L^2(Q_{T_0}),
\\[1mm]
&(A(y_{n_j}))_x\to v&&\quad\mbox{weakly
in}\ L^2(Q_{T_0}).
\end{aligned}
\label{5.7}\end{align}

However, owing to the nonlinear effects of the functions $A$ and $g$, the convergence established in (\ref{5.7}) is insufficient to conclude that $v=(A(y))_x$.
Moreover, it does not guarantee whether, and in what sense, $g_{n_j}(y_{n_j})$ converges to $g(y)$. To this end, we need to establish an estimate for the weak derivative of $y_n$ with respect to time $t$.

Define the Kirchhoff transformation as
$$
A:\ \ y\to z=\int_0^y a(\tau)d\tau.
$$
Under conditions (\ref{1.2}) and (\ref{1.3}), we see that $A$ is a $C^3(\mathbb R)$ function and the transformation $z=A(y)$ is invertible. Denote the inverse transformation as
$$
y=B(z):=A^{-1}(z).
$$
Thus, (\ref{5.1}) can be rewritten as
\begin{align}\label{5.8}
\left\{\begin{array}{ll} (B(z_n))_{t}-(z_n)_{xx}=u_n g_n(B(z_n))-u_n\theta_n, &\mbox{in}\ \ Q_{T_0},
\\[2mm]
z_n(0,t)=z_n(1,t)=0, &\mbox{in}\ \ (0, {T_0}),
\\[2mm]
z_n{(x,0)}=A(y_{0n}(x)), &\mbox{in}\ \ (0,1).
\end{array}\right.
\end{align}
Multiplying (\ref{5.8}) by $(z_n)_t$ and integrating the resulting relation over $Q_{T_0}$,
we deduce
\begin{equation}\label{5.9}
\begin{split}
&\iint_{Q_{T_0}}B'(z_n)(z_n)_t^2dxds+\frac{1}{2}\int_0^1 (z_n)_x^2(x,T_0)dx
\\[2mm]
=&\iint_{Q_{T_0}}u_n g_n(B(z_n))(z_n)_tdxdt-\iint_{Q_{T_0}}u_n\theta_n (z_n)_t dxdt+\frac{1}{2}\int_0^1(z_n)_x^2(x,0)dx
\end{split}
\end{equation}
Recalling now (\ref{1.2}) and (\ref{1.3}), we estimate
\begin{equation}
\label{5.10}
B'(z_n)=(A^{-1}(z_n))'=\frac{1}{A'(y_n)}=\frac{1}{a(A^{-1}(z_n))}\ge \frac{1}{\kappa}.
\end{equation}
Substituting (\ref{5.10}) into (\ref{5.9}) and using Cauchy's inequality with $\varepsilon$ yields
\allowdisplaybreaks
\begin{equation}\label{5.11}
\begin{split}
&\frac{1}{\kappa}\iint_{Q_{T_0}}(z_n)_t^2dxdt+\frac{1}{2}\int_0^1 (z_n)_x^2(x,T_0)dx
\\[2mm]
\le&\iint_{Q_{T_0}}u_n g_n(B(z_n))(z_n)_tdxdt-\iint_{Q_{T_0}}u_n\theta_n (z_n)_t dxdt+\frac{1}{2}\int_0^1(z_n)_x^2(x,0)dx
\\[2mm]
\le&\left(\iint_{Q_{T_0}}|u_n g_n(B(z_n))|^2dxdt\right)^{1/2}\left(\iint_{Q_{T_0}}(z_n)_t^2dxdt\right)^{1/2}
\\[2mm]
&\ \ +\left(\iint_{Q_{T_0}}|u_n\theta_n|^2dxdt\right)^{1/2}\left(\iint_{Q_{T_0}}(z_n)_t^2dxdt\right)^{1/2}+\frac{1}{2}\int_0^1a^2(y_{0n})(y_{0n})_x^2(x,0)dx
\\
\le&\varepsilon \iint_{Q_{T_0}}(z_n)_t^2dxdt+C(\varepsilon)\Big(K_1\sup_{[-A, A]}|g_n|+K_2\Big)+\frac{\kappa^2}{2}\int_0^1(y_{0n})_x^2(x,0)dx.
\end{split}
\end{equation}
Take $\varepsilon=1/(2\kappa)$ in (\ref{5.11}) to deduce
\begin{equation}
\nonumber
\iint_{Q_{T_0}}(z_n)_t^2dxdt\le A_1,
\end{equation}
where $A_1$ is a constant independent of $n$. Since
$$
\big|(y_n)_t\big|=\big|(B(z_n))_t\big|=\frac{1}{a(A^{-1}(z_n))}|(z_n)_t|\le \frac{1}{\rho}\big|(z_n)_t\big|,
$$
we have
\begin{equation}
\label{5.12}
\iint_{Q_{T_0}}(y_n)_t^2dxdt\le A_1/\rho^2.
\end{equation}
From (\ref{5.5}) and (\ref{5.12}), it follows that there exists a subsequence of $\{y_n\}_{n=1}^\infty$, still denoted by $\{y_{n_j}\}_{j=1}^{\infty}$ for convenience, such that
\allowdisplaybreaks
\begin{align}
\begin{aligned}
&y_{n_j}\rightarrow y &&\quad\mbox{strongly in}\ L^2(Q_{T_0}),
\\[2mm]
&(y_{n_j})_t\rightarrow y_t &&\quad\mbox{weakly in}\ L^2(Q_{T_0}),
\end{aligned}
\label{5.13}
\end{align}
as $j\to +\infty$. Furthermore, by combining (\ref{5.7}), it is readily seen that $y\in \stackrel{\circ}{W}{\hskip-1mm}^{1,1}_2(Q_{T_0})\cap L^\infty(Q_{T_0})$. Note that
$$
|A(y_{n_j})-A(y)|=|a(\delta y_{n_j}+(1-\delta) y)||y_{n_j}-y|\le \kappa |y_{n_j}-y| \ \ \ \mbox{for some}\ \delta\in (0,1).
$$
Clearly, \begin{equation}
\label{5.14}
A(y_{n_j})\rightarrow A(y) \quad\mbox{strongly in}\ L^2(Q_{T_0}),\qquad \mbox{as}\ j\to +\infty.
\end{equation}
Using the definition of the weak derivative together with the third limit in (\ref{5.7}) and (\ref{5.14}), we conclude that $v=(A(y))_x$.

Since the sequence $\{y_n\}_{n=1}^\infty$ is uniformly bounded, with $\|y_n\|_{L^\infty(Q_{T_0})}\le A$, the definitions of $g_n(y_n)$ and $g(y)$ are effectively restricted to the interval $[-A, A]$. On $[-A, A]$,
$g_n$ converges uniformly to $g$ and $g\in C([-A, A])$. Consequently, there exists a constant $K>0$ such that $|g_n(y_n)|\le |K$, $g(y)|\le K$, and $|g_n(y_n)-g(y)|\le 2K$. On the other hand, the strong convergence $y_n\to y$
in $L^2(Q_{T_0})$ implies that $y_n$ converges to $y$ in measure on $Q_{T_0}$, and hence $g_n(y_n)$ converges to $g(y)$ in measure on $Q_{T_0}$. Then, by the Lebesgue dominated convergence theorem, we obtain
\begin{equation}\label{5.15}
g_{n_j}(y_{n_j})\rightarrow g(y)\quad\mbox{strongly in}\ L^2(Q_{T_0}).
\end{equation}

For any $\phi\in W^{1, 1}_2(Q_{T_0})$, the weak form of the equation satisfied by $y_{n_j}$ in $Q_{T_0}$ is
\begin{align}
&\iint_{Q_{t}}[y_{n_j}\phi_s-(A(y_{n_j}))_x\phi_x
+u_{n_j}g_{n_j}(y_{n_j})\phi-u_{n_j}\theta_{n_j}\phi] dxds \nonumber \\
=&\int_0^1 y_{n_j}(x,t)\phi(x,t)dx-\int_0^1y_{n_j}(x,0)\phi(x,0)dx, \quad\forall\ 0<t\le T_0. \nonumber
\end{align}
Passing to the limit through the subsequence $n_j$ we obtain
\begin{align}
&\iint_{Q_{t}}[y\phi_s-a(y)y_x \phi_x +1_\omega
ug(y)\phi-u\theta\phi]
dxds=\int_0^1 y(x,t)\phi(x,t)dx-\int_0^1y_{0}\phi(x,0)dx.\nonumber
\end{align}
Therefore, $y\in W^{1,1}_2(Q_{T_0})\cap L^\infty(Q_{T_0})$ is a solution of the system (\ref{1.1}).

Finally, we prove the uniqueness of the solution to (\ref{1.1}) in $W^{1,1}_2(Q_{T_0})\cap L^\infty(Q_{T_0})$.

Suppose that $y_1$ and $y_2$ are two solutions of (\ref{1.1}). Then $y_1-y_2\in \stackrel{\bullet}{W}{\hskip-1mm}^{1,1}_2(Q_T)\cap L^\infty(Q_{T_0})$, and satisfies the weak form
\allowdisplaybreaks
\begin{equation}\label{5.16}
\begin{split}
&\int_0^1 (y_1(x,{T_0})-y_2(x,{T_0}))\phi(x,{T_0})dx-\iint_{Q_{{T_0}}}(y_1-y_2)\phi_tdxdt
\\
&\quad +\iint_{Q_{{T_0}}}(A(y_1)-A(y_2))_x\phi_xdxdt=\iint_{Q_{{T_0}}}1_\omega u(g(y_1)-g(y_2))\phi dxdt
\end{split}
\end{equation}
for any $\phi\in W^{1,1}_2(Q_{{T_0}})$. Set
\begin{equation}
\phi(x,t)=\int_{{T_0}}^t(A(y_1(x,\tau))-A(y_2(x,\tau)))d\tau.
\label{5.17}
\end{equation}
It is straightforward to check $\phi(\cdot,{T_0})=0$ in $(0,1)$, and $\phi\in \stackrel{\circ}{W}{\hskip-1mm}^{1,1}_2(Q_{T_0})$. Then from (\ref{5.16}) we have
{
\allowdisplaybreaks
\begin{align}
&\iint_{Q_{T_0}}(y_1-y_2)(A(y_1)-A(y_2))dxdt \nonumber
\\[2mm]
&\quad-\iint_{Q_{T_0}}\left[(A(y_1)-A(y_2))_x\int_{T_0}^t(A(y_1(x,\tau))-A(y_2(x,\tau)))_xd\tau\right] dxdt \nonumber
\\[2mm]
=&\iint_{Q_{T_0}}(y_1-y_2)(A(y_1)-A(y_2))dxdt \nonumber
\\[2mm]
&\quad-\frac{1}{2}\iint_{Q_{T_0}}\frac{d}{dt}\left(\int_{T_0}^t(A(y_1(x,\tau))-A(y_2(x,\tau)))_xd\tau\right)^2dxdt  \nonumber
\\[2mm]
\le&\left(\iint_{Q_{T_0}}u^2 |g(y_1)-g(y_2)|^2dxdt\right)^{1/2}  \nonumber
\\[2mm]
&\quad\cdot \left(\iint_{Q_{T_0}}\left(\int_{T_0}^t(A(y_1(x,\tau))-A(y_2(x,\tau)))d\tau\right)^2dxdt\right)^{1/2}. \label{5.18}
\end{align}
}

From the preceding proof, we know that when a smaller value of $\|u\|_{L^\infty(Q_{{T_0}})}$ is chosen, the positive constants
$K_1$ and $K_2$ can also be made smaller. Consequently, $A>0$ becomes smaller as well. For the fixed function $g$ that is Lipschitz continuous on any finite interval,
the corresponding Lipschitz constant on the interval $[-A,A]$ will also become smaller as $A$ decreases. Denote by $L_A$ the Lipschitz constant of the function $g$
on the interval $[-A, A]$. Then for any $s_1,s_2\in [-A, A]$, we have $|g(s_1)-g(s_2)|\le L_A|s_1-s_2|$. By H$\ddot{\mbox{o}}$lder's inequality,
\allowdisplaybreaks
\begin{align}
&\iint_{Q_{{T_0}}}\left(\int_{T_0}^t(A(y_1(x,\tau))-A(y_2(x,\tau)))d\tau\right)^2dxdt
\nonumber
\\[2mm]
\le &\iint_{Q_{{T_0}}}\left[\left(\int^{{T_0}}_t(A(y_1(x,\tau))-A(y_2(x,\tau)))^2d\tau\right)^{1/2}\left(\int^{T_0}_t 1 d\tau\right)^{1/2}\right]^2dxdt
\nonumber
\\[2mm]
\le &{T_0}\iint_{Q_{{T_0}}}\left(\int^{T_0}_0(A(y_1(x,\tau))-A(y_2(x,\tau)))^2d\tau\right)dxdt
\nonumber
\\[2mm]
=&{{T_0}}^2\iint_{Q_{T_0}}(A(y_1(x,t))-A(y_2(x,t)))^2dxdt.\nonumber
\end{align}
Returning to the estimate (\ref{5.18}), we obtain
\allowdisplaybreaks
\begin{equation}\label{5.19}
\begin{split}
&\iint_{Q_{{T_0}}}(y_1-y_2)(A(y_1)-A(y_2))dxdt
\\[2mm]
\le&\varepsilon_0L_A {T_0}\left(\iint_{Q_{{T_0}}}(y_1-y_2)^2dxdt\right)^{1/2}
\left(\iint_{Q_{{T_0}}}(A(y_1)-A(y_2))^2dxdt\right)^{1/2}
\\[2mm]
&-\frac{1}{2}\int_0^1\left(\int_{{T_0}}^0(A(y_1(x,\tau))-A(y_2(x,\tau)))_xd\tau\right)^2dx.
\end{split}
\end{equation}
From $A(y_1)-A(y_2)=a(\delta y_1+(1-\delta)y_2)(y_1-y_2)$ for some $\delta\in (0,1)$, and using the condition (\ref{1.3}), we derive
\begin{equation}\label{5.20}
\begin{split}
\rho\iint_{Q_{{T_0}}}(y_1-y_2)^2dxdt\le &\varepsilon_0L_A {T_0}\kappa \iint_{Q_{{T_0}}}(y_1-y_2)^2dxdt
\\[2mm]
&-\frac{1}{2}\int_0^1\left(\int_{{T_0}}^0(A(y_1(x,\tau))-A(y_2(x,\tau)))_xd\tau\right)^2dx.
\end{split}
\end{equation}
By choosing $\varepsilon_0$ sufficiently small, we can ensure that $\varepsilon_0L_A {T_0}\kappa\le \rho/2$. Thus,
\begin{equation}\label{5.21}
\iint_{Q_{{T_0}}}(y_1-y_2)^2dxdt\le-\frac{1}{\rho}\int_0^1\left(\int_{{T_0}}^0(A(y_1(x,\tau))-A(y_2(x,\tau)))_xd\tau\right)^2dx\le 0,
\end{equation}
and hence $y_1=y_2$ a.e. on $Q_{T_0}$. The proof is complete. \hfill $\Box$
\vskip2mm

{\bf Proof of Theorem 1.3.}\ Given $y_0\in H^1_0(0,1)$ and $\sigma>0$. By Remark 4.1, there exist at least a control $u\in U_\sigma$ and $T=T(\sigma, \theta, y_0)>0$,
such that the solution $y$ to (\ref{1.1}) satisfies $y(\cdot, T(\sigma, \theta, y_0))=0$ a.e. in $(0,1)$. Let us estimate the minimal control time.

Define
\begin{align}
T^*(\sigma)
=\displaystyle\inf\{T(\sigma,\theta, y_0); \ y(\cdot, T(\sigma,\theta,y_0))
=0, \ u\in U_\sigma\}.
\nonumber
\end{align}
By Theorem 1.1, we have $0\le T^*(\sigma)\le T(\sigma,\theta, y_0)<\infty$. Hence $T^*(\sigma)$ is a finite number. To approach this minimal time,
we construct sequences $\{T_n\}_{n=1}^\infty$ and $\{u_n\}_{n=1}^\infty\subset U_\sigma$ such that $T_n\ge T^*(\sigma)$ for all $n\in{\mathbb N}^+$, $T_n\rightarrow T^*(\sigma)$
as $n\rightarrow\infty$, and for each control $u_n$, the corresponding solution $y_n$ of (1.1) satisfies $y_n(\cdot, T_n)=0$ a.e. in $(0,1)$ for all $n\in{\mathbb N}^+$.

Denote
\begin{align}
\tilde{u}_n(x,t)
=\left\{\begin{array}{ll}
u_n(x,t) &\ x\in (0,1),\ 0\le t\le T_n,
\\[1mm]
0 &\ x\in (0,1),\ t>T_n.
\end{array}
\right.
\label{5.22}
\end{align}
Let $\tilde{y}_n$ be the solution of (\ref{1.1})with $u=\tilde{u}_n$.
Then $\tilde{y}_n(\cdot, T_n)=0$ a.e. in $(0,1)$. Given $T>T^*(\sigma)$,
we can take $n_0\in{\mathbb N}^+$ such that $T\ge T_n$ for
all $n\ge n_0$. Since $\|\tilde{u}_n\|_{L^\infty(Q_T)}\le \sigma$, there exist a
subsequence of $\{\tilde{u}_n\}_{n=1}^{\infty}$, denoted also by itself, and a
function $\tilde{u}^*\in L^\infty(Q_T)$ such that
\begin{equation}
\tilde{u}_n\rightarrow\tilde{u}^*\quad \mbox{weakly star in}\ L^\infty(Q_T)
\label{5.23}
\end{equation}
with $\|\tilde{u}^*\|_{L^\infty(Q_T)}\le\sigma$.

Similar to the estimates for $u_n$ in Lemma 5.1, we can prove that $\{\tilde{y}_n\}_{n=1}^\infty$ is bounded in
$\stackrel{\circ}{W}{\hskip-1.5mm}^{1,1}_2(Q_T)\cap L^\infty(Q_T)$. Thus,
we can draw out a subsequence of
$\{\tilde{y}_n\}_{n=1}^\infty$, still denoted in the same way, and a limit
function $\tilde{y}^*$, such that
\allowdisplaybreaks
\begin{align}
\begin{aligned}
&\tilde{y}_n\rightarrow \tilde{y}^* &&\quad\mbox{weakly star in}\ L^\infty(Q_T),
\\[1mm]
&\tilde{y}_n\rightarrow \tilde{y}^* &&\quad\mbox{weakly in}\ \stackrel{\circ}{W}{\hskip-1.5mm}^{1,0}_2(Q_T),
\\[1mm]
&\tilde{y}_n\rightarrow \tilde{y}^* &&\quad\mbox{strongly in}\ C([0,T]; L^2(0,1)),
\\[1mm]
&(A(\tilde{y}_n))_x\to A(\tilde{y}^*)_x &&\quad\mbox{weakly in}\ L^2(Q_{T}),
\\[1mm]
&g(\tilde{y}_{n})\rightarrow g(\tilde{y}^*) &&\quad\mbox{strongly in}\ L^2(Q_{T}).
\end{aligned}
\label{5.24}
\end{align}
Returning to the weak formulation satisfied by $\tilde{y}_n$, we see that for any $\phi\in W^{1, 1}_2(Q_T)$,
\begin{align}
&\iint_{Q_{t}}[\tilde{y}_{n}\phi_s-(A(\tilde{y}_{n}))_x\phi_x
+1_\omega\tilde{u}_n g(\tilde{y}_{n})\phi-1_\omega \tilde{u}_n\theta\phi] dxds \nonumber \\
=&\int_0^1 \tilde{y}_{n}(x,t)\phi(x,t)dx-\int_0^1 y_0(x)\phi(x,0)dx, \quad\forall\ 0<t\le T\nonumber.
\end{align}
Letting $n\to+\infty$, we obtain
\begin{align}
&\iint_{Q_{t}}[\tilde{y}^*\phi_s-(A(\tilde{y}^*))_x \phi_x +1_\omega u(g(\tilde{y}^*)-\theta)\phi]
dxds=\int_0^1 \tilde{y}^*(x,t)\phi(x,t)dx-\int_0^1y_{0}\phi(x,0)dx.\nonumber
\end{align}
Therefore, $\tilde{y}^*\in W^{1,1}_2(Q_{T})\cap L^\infty(Q_{T})$ is a solution of the following problem:
\begin{align}
\left\{\begin{array}{ll} \tilde{y}^*_{t}-(a(\tilde{y}^*)\tilde{y}^*_x)_x=1_\omega\tilde{u}^*(g(\tilde{y}^*)-\theta) &(x,t)\in (0,1)\times(0,T),
\\[1mm]
\tilde{y}^*(1,t)=\tilde{y}^*(0,t)=0 &t\in (0,T),
\\[1mm]
\tilde{y}^*(x,0)=y_0(x) &x\in (0,1).\end{array}\right.
\nonumber
\end{align}
Moreover, by the third limit in (\ref{5.24}),
\begin{align}
&\|\tilde{y}^*(\cdot, T^*(\sigma))-\tilde{y}_n(\cdot,T_n)\|_{L^2(0,1)}
\nonumber
\\[1mm]
\le&\|\tilde{y}^*(\cdot, T^*(\sigma))-\tilde{y}_n(\cdot,T^*(\sigma))\|_{L^2(0,1)}
+\|\tilde{y}_n(\cdot, T^*(\sigma))-\tilde{y}_n(\cdot,T_n)\|_{L^2(0,1)}
\rightarrow 0
\nonumber
\end{align}
as $n\rightarrow+\infty$, which implies $\tilde{y}^*(\cdot,T^*(\sigma))=0$.
Finally, denote
\begin{align}
u^*(x,t)
=\left\{\begin{array}{ll}
\tilde{u}^*(x,t) &\ x\in (0,1),\ 0\le t\le T^*(\sigma),
\\[2mm]
0 &\ x\in (0,1),\ t>T^*(\sigma),
\end{array}
\right.
\label{5.25}
\end{align}
and let $y^*$ be the solution to (\ref{1.1}) with $u=u^*$. Then
$y^*(\cdot, T^*(\sigma))=0$
a.e. in $(0,1)$, and hence $T^*(\sigma)$ is the optimal time of
the problem (P). The proof is complete.\hfill $\Box$

\end{document}